\documentclass[11pt]{article}

\usepackage{tcolorbox}
\usepackage{graphicx}
\usepackage{amsmath,amsthm,enumerate,amssymb}
\usepackage{mathtools}
\usepackage{euscript,mathrsfs}
\usepackage{dsfont}
\usepackage{empheq}
\usepackage[left=1.5cm,right=1.5cm,top=3.5cm,bottom=3.5cm]{geometry}
\catcode`\@=11 \@addtoreset{equation}{section}

\catcode`\@=12

\tcbset{arc=0mm,colframe=white,colback=black!10!white}

\theoremstyle{definition}
\newtheorem{Theorem}{Theorem}[section]
\newtheorem{Proposition}[Theorem]{Proposition}
\newtheorem{Assumptions}[Theorem]{Assumptions}
\newtheorem{Lemma}[Theorem]{Lemma}
\newtheorem{Corollary}[Theorem]{Corollary}

\newtheorem{Definition}[Theorem]{Definition}

\newtheorem{Remark}[Theorem]{Remark}

\newcommand{\bTheorem}[1]{
	\begin{Theorem} \label{T#1} }
	\newcommand{\eT}{\end{Theorem}}

\newcommand{\bProposition}[1]{
	\begin{Proposition} \label{P#1}}
	\newcommand{\eP}{\end{Proposition}}

\newcommand{\bLemma}[1]{
	\begin{Lemma} \label{L#1} }
	\newcommand{\eL}{\end{Lemma}}
\newcommand{\bCorollary}[1]{
	\begin{Corollary} \label{C#1} }
	\newcommand{\eC}{\end{Corollary}}

\renewcommand{\(}{\left(}
\renewcommand{\)}{\right)}

\newcommand{\cblue}{\color{blue}}

\allowdisplaybreaks

\renewcommand{\u}{\mathbf{u}}
\newcommand{\U}{\mathbf{U}}
\newcommand{\vv}{\mathbf{v}}

\newcommand{\w}{\mathbf{w}}

\renewcommand{\b}{\mathbf{b}}
\newcommand{\dta}{\tau}
\newcommand{\ddta}{\,\mathrm{d}\tau}
\newcommand{\iQt}{\int_{\Omega_t}}

\newcommand{\CC}{\mathbf{C}}
\newcommand{\HH}{\mathbf{H}}
\newcommand{\DD}{\mathbf{D}}

\newcommand{\BB}{\mathbf{B}}
\newcommand{\I}{\mathbf{I}}
\newcommand{\TT}{\mathbf{T}}
\newcommand{\tr}[1]{\mathrm{tr}\({#1}\)}
\newcommand{\trr}[1]{\mathrm{tr}({#1})}
\newcommand{\trC}{\mathrm{tr}(\CC)}
\newcommand{\trH}{\mathrm{tr}(\HH)}
\newcommand{\Du}{\mathrm{D}\u}
\renewcommand{\div}[1]{\mathrm{div}\left({#1}\right)}
\newcommand{\di}[1]{\mathrm{div}\,{#1}}
\newcommand{\dib}[1]{\mathrm{div}\big({#1}\big)}
\newcommand{\p}{\partial}
\newcommand{\na}{\nabla}

\newcommand{\dx}{\,\mathrm{d}x}

\renewcommand{\d}{d}

\DeclarePairedDelimiter{\norm}{\|}{\|}
\DeclarePairedDelimiter{\snorm}{|}{|}

\newcommand\restr[2]{\ensuremath{\left.#1\right|_{#2}}}

\def\softd{{\leavevmode\setbox1=\hbox{d}%
		\hbox to 1.05\wd1{d\kern-0.4ex{\char039}\hss}}}%cstocs

\date{}

\providecommand{\keywords}[1]
{
	\small	
	\textbf{Keywords:} #1
}

\date{\today}

\title{Relative energy and weak-strong uniqueness of a two-phase viscoelastic phase separation model}
\author{Aaron Brunk \and M\' aria Luk\' a\v cov\' a-Medvi{\softd}ov\'a}

\begin{document}

\maketitle

\bigskip

\centerline{ Institute of Mathematics, Johannes Gutenberg-University Mainz}

\centerline{Staudingerweg 9, 55128 Mainz, Germany}

\centerline{abrunk@uni-mainz.de}
\centerline{lukacova@uni-mainz.de}

\begin{abstract}The aim of this paper is to analyze a viscoelastic phase separation model.
	We derive a suitable notion of the relative energy taking into account the non-convex nature of the energy law for the viscoelastic phase separation. This allows us to prove the  weak-strong uniqueness principle. We will provide the estimates for the full model in two space dimensions. For a reduced model we present the estimates in three space dimensions and derive conditional relative energy estimates.
\end{abstract}


\maketitle

\section{Introduction}

Phase separation is an important process for many physical and industrial applications. Newtonian binary mixtures are well understood and modelled by "model H" \cite{Hohenberg}, which is given by the Cahn-Hilliard-Navier-Stokes system. On the other hand, transition to dynamically asymmetric constituents, for example  polymer-solvent mixtures, is much less understood.  As reported in  \cite{Tanaka., Zhou.2006} many interesting new phenomena such as phase inversion, transient formation of network-like structures and volume shrinking, arise due to the dynamic asymmetry of the mixture and the viscoelastic behaviour. The dynamic asymmetry\cite{Tanaka2017} follows from the different time-scales of the polymer and the solvent.  For such problems the term \emph{viscoelastic phase separation} was coined by Tanaka \cite{Tanaka.}.

The mathematical model for viscoelastic phase separation was later reconsidered by Zhou, Zhang, E \cite{Zhou.2006}, where a thermodynamically consistent version of the viscoelastic phase separation model has been derived. In \cite{Brunk.b,Brunk.} we have proven the existence of a global weak solution to a viscoelastic phase separation model.

Our model of viscoelastic phase separation describes  time evolution of the volume fraction of a polymer $\phi$ and the bulk stress $q\I$, which is a pressure arising from (microscopic) intermolecular attractive interactions. This  leads to a strongly coupled cross-diffusion system. The evolution of $\phi$ is governed by the Cahn-Hilliard type equation, while the evolution of the  bulk stress by a convection-diffusion-reaction equation. These equations are further combined with the Navier-Stokes-Peterlin system for the flow velocity $\u$ and the conformation tensor $\CC$, which describes viscoelastic effects of a polymeric phase, see \cite{2a} for physical details of model derivation.

Through the last decade the concept of relative entropy or relative energy has gained much attention in physical and mathematical literature. We should mention in particular the seminal works of Dafermos \cite{Dafermos.1979,Dafermos.1979b} and DiPerna \cite{DiPerna.}.  Taking into account the physical nature of nonlinear problems together with the underlying energetic structure we can derive a problem-intrinsic metric-like tool to compare different solutions.

The weak-strong uniqueness principle is a consequence of the relative entropy inequality and can be freely stated as follows. Consider a weak and a more regular solution (e.g. strong or classical solution) of a problem. Then the relative energy measuring a distance between both solutions at a given time can be bounded by the relative energy at the initial time. The latter vanishes if both solutions start from the same initial data \cite{Wiedmann.2019}.

The above concepts can be used to find a required regularity of generalized solutions in order to investigate uniqueness of a solution. The concept is often used in the literature for hyperbolic conservation laws \cite{Dafermos.1979,tzavaras} and compressible Navier-Stokes equations \cite{Feireisl.2018,Feireisl.2012,Lu.2018,21,22,23}. A common property of these models is the convexity of the energy law.  Due to the phase separation behavior that we study in the present paper we need a new approach taking into account a non-convex energy.  In \cite{Feireisl.2016,Hosek.2019,lasarzik2019weak,lasarzik2020analysis} the Allen-Cahn equation describing the phase-field process has been investigated applying the relative energy, see also \cite{Giesselmann.2017} for further reading. Application of this technique to viscoelastic equations can be found in \cite{Lu2018,brunk2021existence}. Maximal dissipative solutions applying the weak-strong uniqueness were analysed in \cite{lasarzik2020maximal}. \\

The aim of this paper is to develop a relative energy concept for the viscoelastic phase separation model. We will further derive the weak-strong uniqueness principle for the viscoelastic phase separation model. For the two-dimensional case the weak-strong uniqueness result is unconditional. Furthermore, we state the corresponding results in three space dimensions requiring the existence of a global weak solution. Finally, by combination with the results of \cite{brunk2021existence} we derive a conditional weak-strong uniqueness principle for the full model in three space dimensions.

\section{Mathematical Model}
The total energy of the polymer-solvent mixture consists of the mixing energy between the polymer and the solvent, the bulk stress energy, the elastic energy and the kinetic energy\cite{Brunk.b,Brunk.}.
\begin{align}
	E_{tot}(\phi,q,\CC,\u)&=  E_{mix}(\phi) + E_{bulk}(q) + E_{el}(\CC) + E_{kin}(\u) \label{eq:free_energy}\\
	&=\int_\Omega \(\frac{c_0}{2}\snorm*{\na\phi}^2 + F(\phi)\) \dx + \int_\Omega \frac{1}{2}q^2 \dx + \int_\Omega \(\frac{1}{4}\tr{\TT - 2\ln(\CC) - \I}\)\dx + \int_\Omega \frac{1}{2}\snorm*{\u}^2 \dx. \nonumber
\end{align}
Here $\phi$ denotes the volume fraction of polymer molecules, $q\I$ the bulk stress arising from polymeric interactions, $\CC$ the viscoelastic conformation tensor and $\u$ the volume averaged velocity consisting of a solvent and a polymer velocity. Furthermore, $c_0$ is a positive constant controlling the interface width and $F(\phi)$ is a generally non-convex mixing potential. The viscoelastic phase separation model has been derived in \cite{Brunk.b,2a} from (\ref{eq:free_energy}) applying the GENERIC methodology \cite{Grmela.1997,Grmela.1997b} or the virtual work principle \cite{Groot.2016}. { We note that a direct application of the GENERIC approach would yield  non-dissipative time evolution of viscoelastic effects as shown in \cite{Zhou.2006}.
Motivated by the recent works \cite{Barrett}, \cite{Malek}, where dissipative viscoelastic models were introduced, we add in the time evolution of the bulk stress $q\I$ and conformation tensor $\CC$ additional dissipative terms. Being in the isothermal case the dissipative operators act on
the variational derivative of the energy. This directly applies for the dissipative term in the bulk stress equation and can be generalized also for the equation of the conformation tensor $\CC$ using a suitable nonlinear function of
$\frac{\delta E}{\delta \CC}.$
Consequently, the resulting system of partial differential equations is of damped Hamiltonian type and } reads
\begin{tcolorbox}
\vspace{-1em}
	\begin{align}
		\label{eq:full_model}
		\begin{split}
			\frac{\partial \phi}{\partial t} + \hspace{1em}\u\cdot\nabla\phi  &= \dib{m(\phi)\nabla\mu} - \dib{n(\phi)\nabla\big(A(\phi)q\big)} \\
			\frac{\partial q   }{\partial t} + \hspace{1em}\u\cdot\nabla q    &= -\frac{1}{\tau_b(\phi)}q + A(\phi)\Delta\big(A(\phi)q\big) - A(\phi)\dib{n(\phi)\nabla\mu} + \varepsilon_1\Delta q\\
			\frac{\partial\u   }{\partial t} + (\u\cdot\nabla)\u  &= \dib{\eta(\phi)\Du} -\nabla p + \di{\TT}  + \nabla\phi\mu \\
			\frac{\partial\CC  }{\partial t} + (\u\cdot\nabla)\CC &= (\nabla\u)\CC + \CC(\nabla\u)^\top - h(\phi)\trC\,[\TT-\I] + \varepsilon_2\Delta\CC \\
			\div{\u} &= 0 \hspace{1cm}\TT = \tr{\CC}\CC  \hspace{1cm}   \mu = -c_0\Delta\phi + F^\prime(\phi).
		\end{split}
	\end{align}
\end{tcolorbox}
Here $\Du$ denotes the symmetric velocity gradient of $\u.$ System (\ref{eq:full_model}) is formulated on $(0,T)\times\Omega$, where $\Omega\subset\mathbb{R}^d, d=2,3$ with a convex Lipschitz-continuous boundary. It is equipped with the following initial and boundary conditions, respectively,
\begin{align}
	\restr{(\phi,q,\u,\CC)}{t=0} = (\phi_0 , q_0, \u_0, \CC_0), \quad \restr{\p_n\phi}{\p\Omega}=\restr{\p_n\mu}{\p\Omega}=\restr{\p_nq}{\p\Omega}=0,  \restr{\u}{\p\Omega} = \mathbf{0}, \restr{\p_n\CC}{\p\Omega} = \mathbf{0}. \label{eq:bc}
\end{align}

In what follows we formulate several assumptions on the model parameters.

\begin{Assumptions}\label{ass:regular}
	\hspace{1em} \\
	\vspace{-1em}
	\begin{itemize}
		\item The functions $m,n,h,\eta,\tau_b,A$ are continuous and positively bounded from above and below by $m_2,m_1$, respectively.
		\item We assume that $n^2(s)=m(s)$.
		\item We assume that  $A\in C^1(\mathbb{R})$  with $\norm*{A^\prime}_{L^\infty(\mathbb{R})}\leq A^\prime, A'\geq 0.$
		\item We assume that $F\in C^2(\mathbb{R})$ with constants $c_{1,i}, c_{2,i} > 0,\; i=1,\ldots,2$ and $c_4\geq 0$ such that:
		\begin{align*}
			|F^{(i)}(x)| \leq c_{1,i}|x|^{p-i} + c_{2,i},\; i=0,1,2 \textnormal{ and } p\geq 2, \quad F(x) \geq -c_3, \quad F^{\prime\prime}(x) \geq - c_4.
		\end{align*}
	\end{itemize}
\end{Assumptions}
An example of such a potential is the well-known Ginzburg-Landau potential $F(\phi)=\phi^2(1-\phi)^2$.
For the relative energy proof we require in addition to Assumption \ref{ass:regular}  the following set of assumptions.
\begin{Assumptions}\label{ass:rel}
	\hspace{1em} \\
	\vspace{-1em}
	\begin{itemize}
		\item The functions $m,n,h,\tau_b,A$ are $C^1(\mathbb{R})$ and the $L^\infty$-norm of their derivatives are bounded.
		\item We assume that $A\in C^2(\mathbb{R})$ with $\norm*{A^{\prime\prime}}_{L^\infty(\mathbb{R})}< \infty$.
		\item We assume that $F\in C^3(\mathbb{R})$ with $ p\leq 4$ and $\snorm*{F^{\prime\prime\prime}(x)}\leq c_{1,3}\snorm*{x}+c_{2,3},\quad c_{1,3},c_{2,3}\geq 0$.
	\end{itemize}
\end{Assumptions}

\section{Preliminaries}

In this section we introduce a suitable notation and recall some well-known analytical tools. The space-time cylinder and the intermediate space-time cylinder are denoted by $\Omega_T:=\Omega\times(0,T), \Omega_t:=\Omega\times(0,t),$ for  $t\in[0,T],$  respectively.
For the standard Lebesgue spaces $L^p(\Omega)$ the norm is denoted by $\norm*{\cdot}_p$. Further, we denote the space of divergence free functions by
\begin{equation*}
	L^2_{\text{div}}(\Omega)^d:= \overline{C_{0,\text{div}}^\infty(\Omega)^d}^{\norm*{\cdot}_2}. %\text{ and }  L^2_0(\Omega):=\left\{u \in L^2(\Omega) \mathrel{\bigg|} \int_\Omega u \,\mathrm{d}x  =0\right\},
\end{equation*}
 We use the standard notation for the Sobolev spaces and introduce the notation $	V:=H^1_{0,\text{div}}(\Omega)^d,  H:=L^2_{\text{div}}(\Omega)^d.$
The space $V$ is equipped with the norm $\norm*{\cdot}_V:=\norm*{\na\cdot}_2$.
We denote the norms of the Bochner space $L^p(0,T;L^q(\Omega))$ by $\norm*{\cdot}_{L^p(L^q)}.$

\begin{Lemma}[\cite{LukacovaMedvidova.2015}]
	\label{lemm:hana}
	Let $d=2$ and {$\CC:\Omega\to\mathbb{R}^{2\times 2}$} be a symmetric tensor and let {$\u:\Omega\to\mathbb{R}^2$} be a solenoidal vector field. Then the following identity holds true
	\begin{equation}
		\label{eq:hanastress}
		\trC\CC:\na\u = \frac{1}{2}\Big[(\na\u)\CC + \CC(\na\u)^T \Big]:\CC.
	\end{equation}
\end{Lemma}

\begin{Theorem}[\cite{Brunk.b}]
	\label{theo:ex}
	Given the initial data  $\left(\phi_0,q_0,\u_0,\CC_0 \right) \in [H^1(\Omega)\times L^2(\Omega) \times H \times L^2(\Omega)^{2\times2}].$
	Let Assumptions~\ref{ass:regular} hold. Then for any $T>0$ there exists a global weak solution $(\phi,q,\mu,\u,\CC)$ of (\ref{eq:full_model}) in the following sense	
	\begin{align}
		&\phi \in L^{\infty}(0,T;H^1(\Omega))\cap L^2(0,T;H^3(\Omega)),& & q\I,\CC\in L^\infty(0,T;L^2(\Omega)^{2\times 2})\cap L^2(0,T;H^1(\Omega)^{2\times 2}), \nonumber \\ \nonumber
		&\u\in L^{\infty}(0,T;L^2(\Omega)^2)\cap L^2(0,T;V),& & \Delta\phi,A(\phi)q, \mu\in L^2(0,T;H^1(\Omega)),\\ \nonumber
	\end{align}
	\vspace{-3em}
	\begin{equation}
		\frac{\partial \phi}{\partial t} \in L^2(0,T;(H^{1}(\Omega))^*),\qquad \frac{\partial q\I}{\partial t},\frac{\partial \CC}{\partial t} \in L^{4/3}(0,T;(H^{1}(\Omega)^{2\times 2})^*),\qquad \frac{\partial \u}{\partial t} \in L^2(0,T;V^*), \label{reg1}
	\end{equation}
	and for any test-function $(\psi,\zeta,\xi,\vv,\DD)\in [H^1(\Omega)^3\times V \times H^1(\Omega)^{2\times 2}]$ and a.e. $t\in(0,T)$ we have
	\begin{align}
		&\int_\Omega\frac{\partial \phi}{\partial t}\psi \dx + \int_\Omega\(\u\cdot\na\phi\)\psi\dx + \int_\Omega m(\phi)\nabla\mu\cdot\nabla\psi  \dx - \int_\Omega n(\phi)\nabla\big(A(\phi)q\big)\cdot\na\psi \dx = 0 \nonumber\\
		&\int_\Omega\frac{\partial q   }{\partial t}\zeta \dx + \int_\Omega\(\u\cdot\na q\)\zeta\dx + \int_\Omega\frac{q\zeta}{\tau_b(\phi)}\dx + \int_\Omega\Big(\nabla\big(A(\phi)q\big)- n(\phi)\nabla\mu\Big)\cdot\na\big(A(\phi)\zeta\big)\dx + \int_\Omega \varepsilon_1\nabla q\cdot \nabla\zeta \dx=0 \nonumber\\
		&\int_\Omega\mu\xi \dx - \int_\Omega c_0\nabla\phi\cdot\nabla\xi \dx - \int_\Omega F^{\prime}(\phi)\xi \dx = 0 \label{eq:weak_reg}\\
		&\int_\Omega\frac{\partial\u   }{\partial t}\cdot\vv \dx + \b(\u,\u,\vv) + \int_\Omega\eta(\phi)\nabla\u:\nabla\vv \dx + \int_\Omega\trC\CC:\nabla\vv \dx - \int_\Omega\mu\nabla\phi\cdot\vv \dx = 0 \nonumber\\
		&\int_\Omega\frac{\partial\CC  }{\partial t}:\DD \dx + \BB(\u,\CC,\DD) - \int_\Omega\left[(\nabla\u)\CC + \CC(\nabla\u)^T\right]:\DD \dx + \int_\Omega\varepsilon_2\nabla\CC:\nabla\DD \dx =-\int_\Omega h(\phi)\trC[\trC\CC-\I]:\DD \dx.  \nonumber
	\end{align}
	The weak solution satisfies the initial data $\left(\phi(0),q(0),\u(0),\CC(0) \right) = \left(\phi_0,q_0,\u_0,\CC_0 \right).$  	Moreover the energy inequality
	\begin{align}
		&\left(\int_\Omega \frac{c_0}{2}\snorm{\nabla\phi(t)}^2 + F(\phi(t)) + \frac{1}{2}|q(t)|^2 +\frac{1}{2}\snorm{\u(t)}^2 + \frac{1}{4}\snorm*{\CC(t)}^2 \dx\right) \nonumber \\
		\leq &-\iQt \snorm*{n(\phi)\nabla\mu-\na\Big(A(\phi)q\Big)}^2+ \frac{1}{\tau_b(\phi)}q^2 +\varepsilon_1 \snorm*{\na q}^2 + \eta(\phi)\snorm*{\Du}^2  + \frac{\varepsilon_2}{2}\snorm*{\nabla\CC}^2 + \frac{1}{2} h(\phi)\snorm*{\trC\,\CC}^2 \dx\ddta \label{eq:energyineq}\\
		& + \frac{1}{2}\iQt h(\phi)\snorm*{\trC}^2 \dx\ddta +\left(\int_\Omega \frac{c_0}{2}\snorm{\nabla\phi(0)}^2 + F(\phi(0)) + \frac{1}{2}|q(0)|^2 +\frac{1}{2}\snorm{\u(0)}^2 + \frac{1}{4}\snorm*{\CC(0)}^2 \dx\right) \nonumber
	\end{align}
 holds for almost all $t\in(0,T).$
\end{Theorem}

\begin{Remark}[Regularity]
	In \cite{Brunk.} it is shown that $\phi\in L^2(0,T;H^3(\Omega)).$  This a priori bound is obtained from the elliptic regularity theory by observing that $\mu$, $F'(\phi)$ is bounded in $L^2(0,T;H^1(\Omega))$ for a suitable growth of $F'$.
  A sufficient condition is to restrict the potential growth in Assumptions~\ref{ass:regular} to $p\leq 4$. Note that the Ginzburg-Landau potential typically used in physical applications satisfies this condition. \\
 {By a more careful inspection one can show that $\frac{\partial\CC  }{\partial t}\in \Big(L^{2}(0,T;H^1(\Omega)^{2\times 2})\cap L^4(0,T;L^4(\Omega)^{2\times 2})\Big)^*$. Indeed,
 it is clear that $\mathbf{C} \in L^4(0,T; L^4(\Omega)^{2\times 2})$, since
 $L^2(0,T; H^1(\Omega)^{2\times 2}) \cap L^\infty(0,T; L^2(\Omega)^{2\times 2}) \hookrightarrow L^4(0,T; L^4(\Omega)^{2\times 2}).$
 Consequently, applying the H\"older inequality we obtain that the term $\| (\nabla \mathbf{u}) \mathbf{C}\|_{L^{4/3}(L^{4/3})}$ is bounded as well. Thus, by a comparison in  equation
 \eqref{eq:weak_reg}$_5$ we observe that
 $$
 \frac{\partial \CC}{\partial t} \in  \Big(L^{2}(0,T; (H^1(\Omega)^{2\times 2})^*) \cap L^{4/3}(0,T;L^{4/3}(\Omega)^{2\times 2})\Big) =
 \Big(L^{2}(0,T;H^1(\Omega)^{2\times 2})\cap L^4(0,T;L^4(\Omega)^{2\times 2})\Big)^*.
 $$ }
 \end{Remark}

\begin{Remark}
Let us explain the conceptual difference between the total energy $E_{tot}$ from \eqref{eq:free_energy} and the energy functional
$\int_\Omega \tfrac{c_0}{2}\snorm{\nabla\phi(t)}^2 + F(\phi(t)) + \tfrac{1}{2}|q(t)|^2 +\tfrac{1}{2}\snorm{\u(t)}^2 + \tfrac{1}{4}\snorm*{\CC(t)}^2 \dx$ from \eqref{eq:energyineq}. The total energy $E_{tot}$ dissipates in time, i.e,
\begin{align}
	E_{tot}(t) + &\int_0^t \snorm*{n(\phi)\nabla\mu-\na\big(A(\phi)q\big)}^2+ \frac{1}{\tau_b(\phi)}q^2 +\varepsilon_1 \snorm*{\na q}^2 + \eta(\phi)\snorm*{\Du}^2 \notag\\
	&+ \frac{\varepsilon_2}{2}\snorm{\na\trC}^2 + \frac{\varepsilon_2}{2}\snorm{\CC^{-1/2}\na\CC\CC^{-1/2}}^2 + \frac{1}{2}h(\phi)\trC^2\mathrm{tr}[\TT+\TT^{-1}-2\I] \leq E_{tot}(0)
\end{align}
 which is consistent with the Second Law of Thermodynamics. On the other hand, this is not necessarily  true for the functional in \eqref{eq:energyineq}, {due to the term $h(\phi)|\trC|^2$ on the right-hand side. By inspection of the dissipation term of the total energy one can see that this term is part of a more complex dissipation mechanism involving the stress tensor $\TT$. Hence, this functional plays the role of an additional Lyapunov functional for a two-dimensional model \eqref{eq:full_model}.} % but is not valid in three space dimensions.
Note that in order to define properly the total energy $E_{tot}$  the positive definiteness of $\CC$ is required, while it is not needed  for the energy in \eqref{eq:energyineq}. Consequently, in two-dimensional model \eqref{eq:full_model} we may profit from    an additional energy inequality \eqref{eq:energyineq} in order to derive suitable a priori estimates and prove the existence of the corresponding weak solutions \cite{Brunk.b, 2a}. We refer a reader to our recent  work \cite{brunk2021existence}, where we have studied positive definiteness of $\CC$ for three-dimensional Navier-Stokes-Peterlin system, i.e. the Navier-Stokes equations with the evolution equation \eqref{eq:full_model}$_4$ for the conformation tensor $\CC$,   and proved the existence and conditional weak-strong uniqueness result.
\end{Remark}

\begin{Remark}\label{rem:3dweak}
In three space dimensions we expect a similar existence result to hold true for the viscoelastic phase separation model.
To prove the existence of weak solutions
to three-dimensional model \eqref{eq:full_model} we can only apply  energy inequality
\eqref{eq:free_energy} instead of \eqref{eq:energyineq}. For three-dimensional Navier-Stokes-Peterlin model the existence of weak solutions  was proven in \cite{brunk2021existence}.
Combining the techniques from \cite{brunk2021existence} with those of \cite{Brunk.b} we expect that the existence proof can be extended to the rest submodel for $(\phi,q,\mu)$.
Consequently, we assume the existence of weak solutions with the following regularity for $0<\delta\ll 1$
	\begin{align*}
		&\phi \in L^{\infty}(0,T;H^1(\Omega))\cap L^2(0,T;H^2(\Omega)),& & q\I,\CC\in L^\infty(0,T;L^2(\Omega)^{2\times 2})\cap L^2(0,T;H^1(\Omega)^{2\times 2}),\\
		&\u\in L^{\infty}(0,T;H)\cap L^2(0,T;V),& & \Delta\phi,A(\phi)q, \mu\in L^{5/3}(0,T;W^{1,5/3}(\Omega)),\\
		&n(\phi)\na\mu-\na(A(\phi)q) \in L^2(0,T;L^2(\Omega)),
	\end{align*}
	\vspace{-1em}
	\begin{align*}
		\frac{\partial \phi}{\partial t} &\in L^2(0,T;(H^{1}(\Omega))^*),& \frac{\partial q}{\partial t} &\in L^{1+\delta}(0,T;(W^{1,5/2}(\Omega))^*),& \frac{\partial \CC}{\partial t} &\in L^{4/3}(0,T;(H^{1}(\Omega)^{3\times 3})^*),& \frac{\partial \u}{\partial t} \in L^{4/3}(0,T;V^*).
	\end{align*}

The loss of regularity of $q$ and correspondingly of the higher derivative of $\phi$ is expected, since the uniform bounds for $\na\mu$ follow from the bounds of $\na(A(\phi)q)$. The latter are weaker due to different exponents depending on the space dimension in the interpolation inequalities. A rigorous proof of the existence of weak solutions for three-dimensional model \eqref{eq:full_model} goes beyond the scope of the present paper and is a topic of our future work.	
\end{Remark}

\begin{Definition}
	\label{defn:moreweak}
	A quadruple $(\psi,Q,\pi,\U,\HH)$ is called a \emph{more regular weak solution} if it is a weak solution, cf.~\eqref{reg1}, \eqref{eq:weak_reg}, and additionally enjoys the regularity
	\begin{align}
%		\psi &\in L^2(0,T;H^3(\Omega)),& \quad \frac{\partial\psi}{\partial t} &\in L^2(0,T;(H^{1}(\Omega))^*),\nonumber \\
		\HH &\in L^4(0,T;L^\infty(\Omega)^{2\times 2}), \nonumber\\
		Q &\in L^4(0,T;L^\infty(\Omega))\cap L^4(0,T;W^{1,6}(\Omega)), &\quad \frac{\partial Q}{\partial t} &\in L^2(0,T;(H^{1}(\Omega))^*),\nonumber\\
		\U &\in L^4(0,T;L^\infty(\Omega)^2)%\cap L^2(0,T;V)
\cap L^2(0,T;W^{1,3}(\Omega)^2), &\quad \frac{\partial\U}{\partial t} &\in L^2(0,T;V^*),
	\end{align}
	\vspace{-2em}
	\begin{align*}
		\int_0^T &
		\norm*{n(\psi)\na\pi - \na\big(A(\psi)Q \big) }_4^4 +\norm*{\mathrm{div}\big[n(\psi)\na\pi - \na\big(A(\psi)Q\big)\big]}_2^2 + \norm*{\div{\U\psi-m(\psi)\na\pi+n(\psi)\na\big(A(\psi)Q\big)}}_2^2 \d t \leq C.
	\end{align*}

\end{Definition}

\begin{Definition}
	\label{defn:strong}
	We call $(\phi,q,\u,\CC)$ a \emph{strong solution} of \eqref{eq:full_model} if it is a weak solution, i.e. \eqref{reg1}, (\ref{eq:weak_reg}) holds, and it enjoys the additional regularity
	\begin{align*}
		\phi\in &\,L^\infty(0,T;H^2(\Omega))\cap L^2(0,T;H^4(\Omega)),&\quad \u\in &\,L^\infty(0,T;V)\cap L^2(0,T;H^2(\Omega)^2\cap V), \\
		q\I,\CC\in &\,L^\infty(0,T;H^1(\Omega)^{2\times 2})\cap L^2(0,T;H^2(\Omega)^{2\times 2}),& \quad \mu, A(\phi)q \in &\,L^2(0,T;H^2(\Omega)), \\
		\frac{\partial \phi}{\partial t}, \frac{\partial q}{\partial t},& \frac{\partial \u}{\partial t}, \frac{\partial \CC}{\partial t} \in L^2(0,T;L^2(\Omega)),& \quad \frac{\partial \na\phi}{\partial t} \in &\,L^2(0,T;(H^{1}(\Omega)^2)^*).
	\end{align*}	
\end{Definition}

\begin{Remark}
Clearly, a strong solution in the sense of Definition \ref{defn:strong} is a more regular weak solution, cf. Definition \ref{defn:moreweak}.
\end{Remark}

\begin{Lemma}[Gronwall]
	\label{lem:gronwall}
	Let $t\geq0$ and $f\in L^1(t_0,T)$ be non-negative and $g,\phi$ continuous functions on $[t_0,T]$. If $\phi$ satisfies
	$$
		\phi(t) \leq g(t) + \int_{t_0}^t f(s)\phi(s) \mathrm{d}s \qquad \text{ for all } t\in[t_0,T]
    $$
    then
    $$
		\phi(t) \leq g(t) + \int_{t_0}^t f(s)g(s)\exp{\(\int_s^t f(\tau) \mathrm{d}\tau\)} \mathrm{d}s \qquad \text{ for all } t\in[t_0,T].
	$$
%	If $g$ is moreover non-decreasing then $
%		\phi(t) \leq g(t)\exp{\(\int_{t_0}^t f(\tau) \mathrm{d}\tau\)} \text{ for all } t\in[t_0,T].$
\end{Lemma}

\begin{Proposition}[\cite{Temam.1997}, \cite{LukacovaMedvidova.2015}]
	For any open set $\Omega\subset\mathbb{R}^d, d=2,3,$ it holds that the forms
	\begin{equation*}
		\b(\u,\vv,\w)\equiv \int_\Omega (\u\cdot\na)\vv\cdot \w \dx  \text{ and } \;\; \BB(\u,\CC,\DD)\equiv \int_\Omega (\u\cdot\na)\CC:\DD \dx
	\end{equation*}
	are continuous and trilinear on $V\times V\times V$ and $V\times H^1(\Omega)^{d \times d}\times H^1(\Omega)^{d \times d}$, respectively. Further the following properties hold
	\begin{align*}
		\b(\u,\u,\vv) = - \b(\u,\vv,\u), \;\;\;\;\;\u \in V, \vv\in H^1_0(\Omega)^d,\qquad \BB(\u,\CC,\DD) = -\BB(\u,\DD,\CC),  \;\;\;\;\;\u \in V, \CC,\DD \in H^1(\Omega)^{d \times d}.
	\end{align*}
\end{Proposition}

\section{Relative energy}
In this section we will introduce the notion of relative energy for our viscoelastic phase separation model (\ref{eq:full_model}).  Focusing on the phase-field equation the usual approach in the framework of the Allen-Cahn equation \cite{Hosek.2019} and the Cahn-Hilliard equation \cite{Boyer.1999} is to neglect the potential $F(\phi)$ in $(\ref{eq:energyineq})_1$ and build a relative energy only on the gradient part. However, due to the cross-diffusive coupling of the volume fraction $(\ref{eq:weak_reg})_1$ and the bulk stress equation $(\ref{eq:weak_reg})_2$ in the energy inequality (\ref{eq:energyineq}) this is not applicable in our model.
Furthermore, in the context of phase separation the potential $F$ is non-convex. Similarly to  \cite{LattanzioC.andTzavarasA.E..} the potential can be convexified by adding a quadratic penalty term. \\
We define the relative energy $\mathcal{E}(\phi,q,\u,\CC\vert \psi,Q,\U,\HH)$ of a weak solution $(\phi,q,\u,\CC)$ of  (\ref{eq:weak_reg}) and the functions $(\psi,Q,\U,\HH)$ as
\begin{align}
	&\mathcal{E}(\phi,q,\u,\CC\vert \psi,Q,\U,\HH) = \mathcal{E}_{mix}(\phi\vert\psi) + \mathcal{E}_{bulk}(q\vert Q) + \mathcal{E}_{kin}(\u\vert\U) + \mathcal{E}_{el}(\CC\vert\HH), \label{eq:relen}\\
	&\mathcal{E}_{mix}(\phi\vert\psi) = \int_\Omega \frac{c_0}{2}\snorm*{\na\phi-\na\psi}^2 + F(\phi) - F(\psi) - F^\prime(\psi)(\phi-\psi) + a(\phi-\psi)^2 \dx, \;a \geq 0,\nonumber\\
	&\mathcal{E}_{bulk}(q\vert Q) = \int_\Omega \frac{1}{2}\snorm*{q-Q}^2 \dx,\; \mathcal{E}_{kin}(\u\vert\U) = \int_\Omega \frac{1}{2}\snorm*{\u-\U}^2 \dx,\; \mathcal{E}_{el}(\CC\vert\HH) = \int_\Omega \frac{1}{4}\snorm*{\CC-\HH}^2 \dx.\nonumber
\end{align}

The penalty parameter $a\geq 0$ in $\mathcal{E}_{mix}$ is chosen such that the relative energy is non-negative. Consequently, the  potential together with the penalty term will be convex. Using the Taylor expansion we can find
\begin{equation}
	\label{eq:entay}
	F(\phi) - F(\psi) - F^\prime(\psi)(\phi-\psi) = \frac{F^{\prime\prime}(z)}{2}(\phi-\psi)^2 \geq -\frac{c_4}{2}(\phi-\psi)^2.
\end{equation}
Note that $z$ is a suitable convex combination of $\phi$ and $\psi$. Here we have used the fact that the second derivative of $F$ is bounded from below by $-c_4$ for $c_4>0$, cf. Assumptions~\ref{ass:regular}. Considering the full mixing relative energy $\mathcal{E}_{mix}$ we can find
\begin{equation*}
	\mathcal{E}_{mix}(\phi\vert\psi) \geq \int_\Omega \frac{c_0}{2}\snorm*{\na\phi-\na\psi}^2 + \left(a-\frac{c_4}{2}\right)(\phi-\psi)^2 \dx
\end{equation*}
which is non-negative for $2a\geq c_4$. Furthermore, one can observe that for $2a > c_4$
\begin{equation}
	\mathcal{E}(\phi,q,\u,\CC\vert \psi,Q,\U,\HH) = 0 \Longleftrightarrow \phi=\psi, q=Q, \u=\U, \CC=\CC \text{ a.e. in } \Omega. \label{eq:relativeunique}
\end{equation}

In what follows we formulate several helpful results which will be needed later.

\begin{Lemma}
	\label{lem:easy_control}
	Let Assumptions \ref{ass:regular} hold. {Further, let $\phi, \psi \in  L^\infty(0,T; H^1(\Omega)) \cap L^2(0,T; H^2(\Omega)),$ $Q \in
L^4(0,T; L^\infty(\Omega)) \cap L^4(0,T; W^{1,6}(\Omega)).$}
Then the integral
	\begin{equation*}
		I_1:=\int_0^t\norm*{\na\Big((A(\phi) - A(\psi))Q\Big)}_2^2 \ddta
	\end{equation*}
	can be bounded by the relative energy as follows:
	\begin{align*}
		I_1&\leq c\int_0^t \norm*{\phi-\psi}_6^2\norm*{\na Q}_3^2 + \norm*{\na\phi-\na\psi}_2^2\norm*{Q}_\infty^2 + (\norm*{\na\phi}_3^2+\norm*{\na\psi}_3^2)\norm*{\phi-\psi}_6^2\norm*{Q}^2_\infty \ddta \\
		&\leq c\int_0^t \mathcal{E}_{mix}(\dta)\Big[\norm*{\na Q}_3^2 + \norm*{Q}_\infty^2 + (\norm*{\na\phi}_3^2+\norm*{\na\psi}_3^2)\norm*{Q}^2_\infty\Big]  \ddta.
	\end{align*}
\end{Lemma}
The proof is done be expanding the gradient and applying the Hölder and Young inequalities. {Clearly, due to the assumed regularity of
$\phi, \psi, Q$ the last term in the bracket is $L^1$- integrable in time, which will be used latter to derive the relative energy inequality, cf.~\eqref{eq:rel_gron_a}, \eqref{eq:rel_gron}.}

\begin{Lemma}
	\label{lem:timeid}
	Let $(\phi,q,\u,\CC)$ be a global weak solution satisfying \eqref{reg1}, (\ref{eq:weak_reg}) and $(\psi,Q,\U,\HH)$ arbitrary functions satisfying
	\begin{align}
		\psi &\in L^2(0,T;H^3(\Omega))\cap H^1(0,T;(H^{1}(\Omega))^*),&  Q &\in L^4(0,T;H^1(\Omega))\cap H^1(0,T;(H^1(\Omega))^*),\label{eq:wsutimereg}\\
		\U &\in L^4(0,T;V)\cap H^1(0,T;V^*), &{\HH} &\in L^4(0,T;L^4(\Omega)^{2\times 2})\cap L^{2}(0,T;H^1(\Omega)^{2\times 2}),\nonumber\\
		{ \frac{\partial \HH}{\partial t}} & {\in \Big(L^{2}(0,T;H^1(\Omega)^{2\times 2})\cap L^4(0,T;L^4(\Omega)^{2\times 2})\Big)^*}.& & \nonumber
\end{align}
	Then the following identities hold for a.a. $t,s\in(0,T)$	
	\begin{align*}
		\int_\Omega \na\phi(t)\cdot\na\psi(t) - \na\phi(s)\cdot\na\psi(s) \dx &= -\int_s^t \int_\Omega \frac{\partial \phi(\dta)}{\partial t}\Delta\psi(\dta) + \Delta\phi(\dta)\frac{\partial \psi(\dta)}{\partial t} \dx\,\mathrm{d}\dta, \\
		\int_\Omega \phi(t)\psi(t) - \phi(s)\psi(s) \dx &= \int_s^t \int_\Omega \frac{\partial \phi(\dta)}{\partial t}\psi(\dta) + \phi(\dta)\frac{\partial \psi(\dta)}{\partial t} \dx\,\mathrm{d}\dta, \\
		\int_\Omega q(t) Q(t) - q(s)Q(s) \dx &= \int_s^t \int_\Omega \frac{\partial q(\dta)}{\partial t}Q(\dta) + q(\dta)\frac{\partial Q(\dta)}{\partial t} \dx\,\mathrm{d}\dta, \\
		\int_\Omega \u(t)\cdot\U(t) - \u(s)\cdot\U(s) \dx &= \int_s^t \int_\Omega \frac{\partial \u(\dta)}{\partial t}\cdot\U(\dta) + \u(\dta)\cdot\frac{\partial \U(\dta)}{\partial t} \dx\,\mathrm{d}\dta, \\
		\int_\Omega \CC(t):\HH(t) - \CC(s):\HH(s) \dx &= \int_s^t \int_\Omega \frac{\partial \CC(\dta)}{\partial t}:\HH(\dta) + \CC(\dta):\frac{\partial \HH(\dta)}{\partial t} \dx\,\mathrm{d}\dta. 	
	\end{align*}
\end{Lemma}	
\begin{proof}
	\rm{One can follow the procedure in \cite{Emmrich.2018} to prove such integration by parts formulas by approximation with a smooth functions and corresponding limiting process.}
\end{proof}

The following result for the relative energy holds for the viscoelastic phase separation model \eqref{eq:full_model} in two space dimensions.
\begin{tcolorbox}
	\begin{Theorem}[Relative energy, $d=2$]
		\label{theo:relative_energy}
		Let $(\phi,q,\mu,\u,\CC)$ be a global weak solution satisfying \eqref{reg1}, (\ref{eq:weak_reg}) starting from the initial data $(\phi_0,q_0,\u_0,\CC_0)$. Let $(\psi,Q,\pi,\U,\HH)$ be a more regular weak solution in the sense of Definition \ref{defn:moreweak} on $(0,T^\dagger); T\leq T^\dagger,$ starting from the initial data $(\psi_0,Q_0,\U_0,\HH_0)$. Further, let Assumptions~\ref{ass:rel} hold.\\
		Then the relative energy given by (\ref{eq:relen}) satisfies the inequality
		\begin{align}
			\mathcal{E}(t) + b\mathcal{D} \leq \mathcal{E}(0) +  \int_0^t g(\dta)\mathcal{E}(\dta) \ddta \label{eq:relgron}
		\end{align}
		for almost all $t\in(0,T^\dagger)$. Here $g\in L^1(0,T^\dagger)$ is given by (\ref{eq:rel_gron}) and $b>0$. Moreover $\mathcal{D}$ is given by
		\begin{align}
			\mathcal{D} =& \iQt \Big| n(\phi)(\na\mu-\na\pi) - \na\Big(A(\phi)(q-Q) \Big) \Big|^2 +  \frac{1}{\tau_b(\phi)}(q-Q)^2  + \varepsilon_1 \snorm*{\na q - \na Q}^2 \dx\ddta\label{eq:Dfull}\\
			+&\iQt \eta(\phi)\snorm*{\Du - \mathrm{D}\U}^2  + \frac{1}{2}h(\phi)\trC^2(\CC-\HH)^2  + \frac{\varepsilon_2}{2}\snorm*{\na\CC-\na\HH}^2 \dx\ddta. \nonumber
		\end{align}	
	\end{Theorem}
	\vspace{-1em}
\end{tcolorbox}

\begin{Remark}
	The proof of Theorem \ref{theo:relative_energy} will be presented in Section 5 and consists of the following steps:
	\begin{enumerate}
		\item Suitable decomposition of the relative energy into the energy of the weak solution, the energy of the strong solution and a remainder term.
		\item Expansion of the remainder term by inserting the weak solution as test function for the strong solutions and vice versa.
		\item Derivation of the relative dissipation functional and estimates of all remaining terms by means of the relative energy and the relative dissipation to derive a Gronwall-type structure. Application of the Gronwall inequality.
	\end{enumerate}
\end{Remark}

Up to our best knowledge the question of relative energy estimates for full three-dimensional viscoelastic phase separation model (\ref{eq:full_model}) is an open problem in general, cf.~\cite{brunk2021existence}.  In what follows we formulate the result for the relative energy of the so-called \emph{reduced viscoelastic phase separation model} in three space dimensions. Hereby, we assume that viscoelastic effects are restricted only to the bulk stress $q\I$ omitting the conformation tensor $\CC$ from the model, i.e.
\begin{align}
	\label{eq:red_model}
	\begin{split}
		&\frac{\partial \phi}{\partial t} + \u\cdot\nabla\phi  = \dib{m(\phi)\nabla\mu} - \dib{n(\phi)\nabla\big(A(\phi)q\big)} \\
		&\frac{\partial q   }{\partial t} + \u\cdot\nabla q    = -\frac{1}{\tau_b(\phi)}q + A(\phi)\Delta\big(A(\phi)q\big) - A(\phi)\dib{n(\phi)\nabla\mu} + \varepsilon_1\Delta q\\
		&\frac{\partial\u   }{\partial t} + (\u\cdot\nabla)\u  = \dib{\eta(\phi)\Du} -\nabla p  + \nabla\phi\mu \\
		&\div{\u} = 0 \hspace{1cm}   \mu = -c_0\Delta\phi + F^\prime(\phi).
	\end{split}
\end{align}
For this reduced model the total energy \eqref{eq:free_energy} and the energy functional \eqref{eq:energyineq} coincide and we denote
\begin{align*}
E_{red}=\int_\Omega \frac{c_0}{2}\snorm{\na\phi}^2 + F(\phi) + \frac{1}{2}q^2 + \frac{1}{2}\snorm{\u}^2.
\end{align*}
Of course, this implies a reduced relative energy $\mathcal{E}_{red}$ by the same arguments as in \eqref{eq:relen}.
\begin{tcolorbox}
	\begin{Theorem}[Relative energy, reduced model, $d=3$]
		\label{theo:relative_energy3d}
		Let $(\phi,q,\mu,\u)$ be a global weak solution of the reduced viscoelastic phase separation model starting from the initial data $(\phi_0,q_0,\u_0)$. Let $(\psi,Q,\pi,\U)$ be a corresponding smooth solution on $(0,T^\dagger); T^\dagger\leq T,$ starting from the initial data $(\psi_0,Q_0,\U_0)$. Furthermore, let Assumptions~\ref{ass:rel} hold.\\
		Then the relative energy given in (\ref{eq:relen}) satisfies the inequality
		\begin{align}
			\mathcal{E}_{red}(t) + b\mathcal{D} \leq \mathcal{E}_{red}(0) +  \int_0^t g(\dta)\mathcal{E}_{red}(\dta) \ddta \label{eq:relgron3}
		\end{align}
		for almost all $t\in(0,T^\dagger)$. Here $g\in L^1(0,T^\dagger),$ see (\ref{eq:rel_gron}) for a precise expression, and $b>0$. Moreover $\mathcal{D}$ is given by
		\begin{align}
			\mathcal{D} &= \iQt \Big| n(\phi)(\na\mu-\na\pi) - \na\Big(A(\phi)(q-Q) \Big) \Big|^2  +  \frac{1}{\tau_b(\phi)}(q-Q)^2+ \varepsilon_1\snorm*{\na q - \na Q}^2 + \eta(\phi)\snorm*{\Du - \mathrm{D}\U}^2 \dx\ddta. \label{eq:Dfull3}
		\end{align}	
	\end{Theorem}
\end{tcolorbox}

\begin{proof}
	\rm{In order to prove this result we will use only estimates for $\phi,\mu,q,\u$ and $\psi,\pi,Q,\U$ that are valid also in three space dimensions. The desired result then follows by neglecting the contributions of the conformation tensors $\CC$ and $\HH$ in the proof of Theorem~\ref{theo:relative_energy}.}
\end{proof}
\begin{Remark}
	For convenience we assumed that $(\psi,Q,\pi,\U)$ is a smooth solution.
%Indeed, it is obvious that the regularity of solutions from Definition \ref{defn:moreweak} does not suffice, since we expect for a weak solution
%only the regularity given in Remark \ref{rem:3dweak}.
In our recent paper \cite{brunk.new} it was shown that the required regularity of a smooth solution of a three-dimensional reduced viscoelastic phase separation model \ref{eq:red_model} is
{
	\begin{align}
	\psi &\in L^\infty(0,T;H^1(\Omega))\cap L^2(0,T;H^3(\Omega)),& \quad \frac{\partial\psi}{\partial t} &\in L^{5/2}(0,T;(W^{1,5/3}(\Omega))^*),\nonumber \\
	Q &\in L^\infty(0,T;L^2(\Omega)) \cap L^4(0,T;L^\infty(\Omega))\cap L^4(0,T;W^{1,6}(\Omega)), &\quad \frac{\partial Q}{\partial t} &\in L^2(0,T;(H^{1}(\Omega))^*),\nonumber\\
	\U &\in  L^\infty(0,T;H) \cap L^4(0,T;L^\infty(\Omega)^2)%\cap L^2(0,T;V)
	\cap L^2(0,T;W^{1,3}(\Omega)^2), &\quad \frac{\partial\U}{\partial t} &\in L^2(0,T;V^*),
\end{align}
\vspace{-2em}
\begin{align*}
	\int_0^T &
	\norm*{n(\psi)\na\pi - \na\big(A(\psi)Q \big) }_4^4 +\norm*{\mathrm{div}\big[n(\psi)\na\pi - \na\big(A(\psi)Q\big)\big]}_2^2 + \norm*{\div{\U\psi-m(\psi)\na\pi+n(\psi)\na\big(A(\psi)Q\big)}}_2^2 \d t \leq C.
\end{align*}}
\end{Remark}

\indent
Further, it was shown in \cite{brunk2021existence} that in order to use the relative energy for three-dimensional Navier-Stokes-Peterlin model, one
has to assume that $\CC\in L^s(0,T;L^r(\Omega))$ with $\frac{2}{s}+\frac{3}{{\cblue r}}\leq 1$ and $2<s<\infty$ and $3<r<\infty.$ Thus, in order to work with the relative energy for full three-dimensional viscoelastic phase separation model (\ref{eq:full_model}) the above additional regularity
for $\CC$ needs to be assumed as it is not known to be fulfilled by a weak solution.  Combination of Theorem~\ref{theo:relative_energy3d} with the relative energy result in \cite{brunk2021existence} implies the following result.
\begin{Corollary}[Conditional relative energy, $d=3$]\label{coro:cond}
Assume that $(\phi,q,\mu,\u, \CC)$ is a global weak solution of the viscoelastic phase separation model \eqref{eq:full_model} starting from the initial data $(\phi_0,q_0,\u_0, \CC_0)$, such that $\CC\in L^s(0,T;L^r(\Omega))$ with $\frac{2}{s}+\frac{3}{{\cblue r}}\leq 1$ and $2<s<\infty$ and $3<r<\infty.$
Let $(\psi, Q, \pi, \U, \HH)$ be sufficiently smooth solution of  \eqref{eq:full_model}
on $(0,T^\dagger); T^\dagger\leq T,$ starting from the initial data $(\psi_0,Q_0,\U_0, \HH_0)$. Furthermore, let Assumptions~\ref{ass:rel} hold.

	Using the modified relative energy
		\begin{align*}
		\mathcal{E}_{3d}&:= \mathcal{E}_{mix}(\phi|\psi) + \mathcal{E}_{bulk}(q|Q) + \mathcal{E}_{kin}(\u|\U) + \frac{1}{4}\norm{\trr{\CC-\HH}}^2 + \beta\mathcal{E}_{el}(\CC|\HH)
	\end{align*}
	for all $\beta>0$,
the following conditional relative energy estimate holds
	\begin{align*}
		\mathcal{E}_{3d}(t) + b\mathcal{D}_{3d} \leq \mathcal{E}_{3d}(0) +  \int_0^t g(\dta)\mathcal{E}_{3d}(\dta) \ddta
	\end{align*}
	for almost all $t\in(0,T^\dagger)$ with $b>0$ and a suitable $\mathcal{D}_{3d}$.
\end{Corollary}

The above relative energy inequalities can be used to obtain the corresponding weak-strong uniqueness principles.

\begin{tcolorbox}
	\begin{Theorem}[Weak-strong uniqueness, $d=2$]
		\label{theo:wsu}
		Let $(\psi,Q,\U,\HH)$ be a strong solution of the viscoelastic phase separation model \eqref{eq:full_model}, cf. Definition \ref{defn:strong}. Furthermore, let the assumptions of Theorem \ref{theo:relative_energy} hold and initial data coincide, i.e. $\phi_0=\psi_0, q_0=Q_0, \u_0=\U_0, \CC_0=\HH_0$.
		Then any weak solution $(\phi,q,\u,\CC)$ coincides with the strong solution $(\psi,Q,\U,\HH)$ almost everywhere in $\Omega\times(0,T^\dagger)$.
	\end{Theorem}
\end{tcolorbox}
From the proof of Theorem~\ref{theo:relative_energy} it will be clear that the above result holds also for the more regular weak solutions in the sense of Definition \ref{defn:moreweak}.

\begin{tcolorbox}
	\begin{Theorem}[Weak-strong uniqueness, $d=3$]
		\label{theo:wsu3d}
		Let $(\psi,Q,\U)$ be a smooth solution of the reduced viscoelastic phase separation model. Furthermore, let the assumptions of Theorem \ref{theo:relative_energy3d} hold and initial data coincide, i.e. $\phi_0=\psi_0, q_0=Q_0, \u_0=\U_0$.
		Then weak solution $(\phi,q,\u)$ coincides with the smooth solution $(\psi,Q,\U)$ almost everywhere in $\Omega\times(0,T^\dagger)$.
	\end{Theorem}
\end{tcolorbox}
\begin{Remark}
	Similarly, using Corollary \ref{coro:cond} we obtain a conditional weak-strong uniqueness result in three space dimensions.
\end{Remark}

\begin{proof}(Theorem~\ref{theo:wsu})\\
	\rm{Using Theorems \ref{theo:relative_energy} or \ref{theo:relative_energy3d} and \ref{lem:gronwall}, and realizing that $\mathcal{E}(0)$ is non-negative and constant in time, we find that $\mathcal{E}(t) \leq \mathcal{E}(0)\exp\(\int_0^t g(\dta) \mathrm{d}\dta\)$
	holds for almost all $t\in(0,T^\dagger)$.
	Moreover, by construction $\mathcal{E}(0)=0$ if and only if $\phi_0=\psi_0, q_0=Q_0, \u_0=\U_0,$ for $d=2,3$ and $\CC_0=\HH_0$ for $d=2$, a.e. in $\Omega$, see (\ref{eq:relativeunique}).}
\end{proof}

\section{Relative Energy Proof}
The aim of this section is to prove Theorem \ref{theo:relative_energy}. First, we will decompose the relative energy into the energy of the weak and the strong solution and the correction terms. Afterwards, the correction terms are expanded by testing the weak solution with suitable functions arising from the more regular solution and vice versa. In the second step we recollect corresponding integrals to derive a structure suitable for the Gronwall lemma, cf. \eqref{eq:relgron}, and estimate the remaining terms.We note in passing that we will estimate all quantities related to $(\phi,\mu,q,u)$ and $(\psi,\pi,Q,\U)$ only using estimates that are also valid in three space dimensions. The contribution of the conformation tensors $\CC$ and $\HH$ are treated by means of special two-dimensional inequalities.
\subsection{Decomposition of the relative energy}
The aim of this section is to decompose the relative energy. For better readability we introduce
\begin{align*}
	z=(\phi,q,\u,\CC), \qquad \bar{z}=(\psi,Q,\U,\HH),\qquad
	z\vert\bar{z}=(\phi,q,\u,\CC)\vert(\psi,Q,\U,\HH).
\end{align*}
We start by decomposing the relative energy into two independent energies and a correction term.
\begin{align}
	\mathcal{E}(z\vert\bar{z})(t) = E(z)(t) + E(\bar{z})(t) - \mathcal{DE}(z\vert\bar{z})(t)&\leq -\tilde{\mathcal{D}} + \mathcal{R} + E(z)(0) + E(\bar{z})(0) - \mathcal{DE}(z\vert\bar{z})(t) \nonumber\\
	&\leq -\tilde{\mathcal{D}} + \mathcal{R} + \mathcal{E}(z\vert\bar{z})(0) + \mathcal{DE}(z\vert\bar{z})(0) - \mathcal{DE}(z\vert\bar{z})(t). \label{eq:decomprelen}
\end{align}
Here we have used the energy inequality (\ref{eq:energyineq}) for the weak and a more regular solution.
Further we denote by $\tilde{\mathcal{D}}$ the dissipative term from the energy inequality, by $\mathcal{R}$ the non-dissipative remainder terms and by $\mathcal{DE}$ the correction terms given by
\begin{align}
	\tilde{\mathcal{D}} =& \iQt \Big|n(\phi)\na\mu - \na\big(A(\phi)q\big)\Big|^2 + \frac{q^2}{\tau_b(\phi)} + \varepsilon_1\snorm*{\na q}^2 + \eta(\phi)\snorm*{\Du}^2 + \frac{1}{2}h(\phi)\trC^2\snorm*{\CC}^2 + \frac{\varepsilon_2}{2}\snorm*{\na\CC}^2 \dx\ddta \nonumber\\
	&+\iQt \Big|n(\psi)\na\pi-\na\big(A(\psi)Q\big)\Big|^2  + \frac{Q^2}{\tau_b(\psi)}+ \varepsilon_1\snorm*{\na Q}^2  + \eta(\psi)\snorm*{\mathrm{D}\U}^2  + \frac{1}{2}h(\psi)\trH^2\snorm*{\HH}^2  + \frac{\varepsilon_2}{2}\snorm*{\na\HH}^2 \dx\ddta, \nonumber\\
	\mathcal{R} =& \frac{1}{2}\iQt h(\phi)\trC^2\dx\ddta + \frac{1}{2}\iQt h(\psi)\trH^2\dx\ddta ,\label{eq:Drel}\\
	\mathcal{DE}(z\vert\bar{z})(t) =& \int_\Omega c_0\na\phi\cdot\na\psi + 2F(\psi) + F^\prime(\psi)(\phi-\psi) -a(\phi-\psi)^2 + qQ + \u\cdot\U + \frac{1}{2}\CC:\HH \dx.\nonumber
\end{align}

Applying Lemma \ref{lem:timeid} to the difference $\mathcal{DE}(0)-\mathcal{DE}(t)$ in (\ref{eq:decomprelen}) we find
\begin{align}
	\mathcal{DE}(0)-\mathcal{DE}(t) =& -\int_0^t \(\frac{\partial}{\partial t}\mathcal{DE}\)(\dta) \mathrm{d}\dta= c_0\iQt \frac{\partial \phi(\dta)}{\partial t}\Delta \psi(\dta) + \frac{\partial \psi(\dta)}{\partial t}\Delta\phi(\dta) \dx\,\mathrm{d}\dta  \label{eq:mixphi}\\
	&- \iQt F^{\prime\prime}(\psi(\dta))\Big(\phi(\dta)-\psi(\dta)\Big)\frac{\partial \psi(\dta)}{\partial t} + F^\prime(\psi(\dta))\(\frac{\partial \phi(\dta)}{\partial t} + \frac{\partial \psi(\dta)}{\partial t}\) \dx\,\mathrm{d}\dta \label{eq:mixpot}\\
	& +2a\iQt  \frac{\partial \phi(\dta)}{\partial t}\Big(\phi(\dta)-\psi(\dta)\Big) + \frac{\partial \psi(\dta)}{\partial t}\Big(\psi(\dta)-\phi(\dta)\Big) \dx\,\mathrm{d}\dta \label{eq:phistab} \\
	&- \iQt \frac{\partial q(\dta)}{\partial t} Q(\dta) + q(\dta)\frac{\partial Q(\dta)}{\partial t} \dx\,\mathrm{d}\dta \label{eq:mixq} \\
	&- \iQt \frac{\partial \u(\dta)}{\partial t}\cdot\U(\dta) + \u(\dta)\cdot\frac{\partial \U(\dta)}{\partial t} \dx\,\mathrm{d}\dta \label{eq:mixu}\\
	&- \frac{1}{2}\iQt \frac{\partial \CC(\dta)}{\partial t}:\HH(\dta) + \CC(\dta):\frac{\partial \HH(\dta)}{\partial t} \dx\,\mathrm{d}\dta. \label{eq:mixC}
\end{align}
In the following subsections we will expand all the above terms suitably. From the obtained terms and the corresponding terms in $\tilde{\mathcal{D}}$, cf. (\ref{eq:Drel}), we will construct $\mathcal{D}$, cf. (\ref{eq:Dfull}).

%%%%%%%%%%%%%%%%%%%%%%%%%%%%%%%%%%%%%%%%%%%%%%%%%%%%%%%%%%%%%%%%%%%%%%%%%%%%%%%%%%%%

\subsection{Navier-Stokes equation}
We start with term (\ref{eq:mixu}) and insert $\u,\U$ as test functions in the weak formulation (\ref{eq:weak_reg}) for $\U,\u$, respectively,
\begin{align*}
	&- \iQt \frac{\partial \u(\dta)}{\partial t}\cdot\U(\dta) + \u(\dta)\cdot\frac{\partial \U(\dta)}{\partial t} \dx\,\mathrm{d}\dta = \sum_{k=1}^4 I_{\u,k} \equiv\\
	\equiv& \iQt \big[\eta(\phi)+\eta(\psi)\big]\Du:\mathrm{D}\U\dx\ddta + \iQt\trC\CC:\na\U + \trH\HH:\na\u \dx\ddta\\
	&- \iQt  \mu\na\phi\cdot\U + \pi\na\psi\cdot\u \dx\ddta + \int_0^t\b(\u,\u,\U) + \b(\U,\U,\u) \ddta.
\end{align*}

Considering the first integral $I_{\u,1}$ and the corresponding terms of the dissipation (\ref{eq:Drel}), i.e. the fourth term $(\ref{eq:Drel})_{1,4}$ of $\tilde D$ and tenth term $(\ref{eq:Drel})_{1,10}$ of $\tilde D$, we find
\begin{align}
	P_1:=&- \iQt \eta(\phi)\snorm*{\Du}^2 + \eta(\psi)\snorm*{\mathrm{D}\U}^2\dx\ddta + \iQt \big[\eta(\phi)+\eta(\psi)\big]\Du:\mathrm{D}\U \dx\ddta\nonumber\\
	=& -\iQt \eta(\phi)|\Du - \mathrm{D}\U|^2\dx\ddta - \iQt \big[\eta(\psi) - \eta(\phi)\big]\mathrm{D}\U:(\mathrm{D}\U - \Du) \dx\ddta\nonumber\\
	\leq& -\iQt \eta(\phi)|\Du - \mathrm{D}\U|^2\dx\ddta + \frac{1}{\sqrt{\eta_1}}\norm*{\eta^\prime}_\infty\int_0^t\norm*{\phi-\psi}_6\norm*{\mathrm{D}\U}_3\norm*{\sqrt{\eta(\phi)}(\Du - \mathrm{D}\U)}_2 \ddta\nonumber\\
	%\leq& -(1-\delta)\iQt \eta(\phi)|\Du - \mathrm{D}\U|^2 \dx\ddta  + c(\eta_1,\eta^\prime,\delta)\int_0^t\norm*{\phi-\psi}_6^2\norm*{\mathrm{D}\U}_3^2 \ddta\nonumber\\
	\leq& -(1-\delta)\iQt \eta(\phi)|\Du - \mathrm{D}\U|^2 \dx\ddta + c(\eta_1,\eta^\prime,\delta)\int_0^t\mathcal{E}(\dta)\norm*{\mathrm{D}\U}_3^2 \ddta. \label{eq:P1}
\end{align}
The convective terms, i.e. $I_{\u,4}$,  by integrating by parts and adding $-\b(\u-\U,\U,\U)=~0$, implies
\begin{align}
	P_2:=&\int_0^t \b(\u,\u,\U) + \b(\U,\U,\u) \ddta =  \int_0^t \b(\u,\u,\U) - \b(\U,\u,\U) \ddta\nonumber\\
	= & \int_0^t \b(\u-\U,\u,\U)- \b(\u-\U,\U,\U) \ddta \leq \,c\int_0^t \norm*{\u-\U}_3 \norm*{\na\u-\na\U}_2 \norm*{\U}_6 \ddta\label{eq:P2}\\
	\leq &\,2\delta\iQt \eta(\phi)|\Du-\mathrm{D}\U|^2 \dx\ddta + c(\delta,\eta_1)\int_0^t\norm*{\u-\U}_2^2\norm*{\U}_6^4 \ddta	\leq \,2\delta\iQt \eta(\phi)|\Du-\mathrm{D}\U|^2 \dx\ddta + c(\delta,\eta_1)\int_0^t\norm*{\U}_6^4 \mathcal{E}(\dta) \ddta.\nonumber
\end{align}
The remaining terms are the coupling terms to the other variables $(\phi,q,\CC)$ and will be treated in what follows.

\subsection{Equation for the conformation tensor}
We recall that all estimates involving the conformation tensor, i.e. $\CC$ and $\HH$, presented in this subsection are done in two space dimensions. We proceed with the time derivatives of the conformation tensor $\CC, \HH$, cf. (\ref{eq:mixC})
\begin{align*}
	&- \frac{1}{2}\iQt \frac{\partial \CC(\dta)}{\partial t}:\HH(\dta) + \CC(\dta):\frac{\partial \HH(\dta)}{\partial t} \dx\ddta = \sum_{k=1}^5 I_{\CC,k} \equiv\\
	\equiv& \varepsilon_2\iQt \na\CC:\na\HH \dx\ddta + \frac{1}{2}\iQt h(\phi)\trC^2\CC:\HH + h(\psi)\trH^2\HH:\CC \dx\ddta\\
	&- \frac{1}{2}\iQt h(\phi)\trC\trH + h(\psi)\trH\trC\dx\ddta + \int_0^t\frac{1}{2}\BB(\u,\CC,\HH) + \frac{1}{2}\BB(\U,\HH,\CC) \ddta\\
	&  - \frac{1}{2}\iQt (\na\u)\CC:\HH + \CC(\na\u)^T:\HH + (\na\U)\HH:\CC + \HH(\na\U)^T:\CC \dx\ddta.
\end{align*}

First we treat the term $I_{\CC,3}$ together with the remainder terms $\mathcal{R}$ of $(\ref{eq:Drel})_2$, i.e the second term $(\ref{eq:Drel})_{2,2}$ and first term $(\ref{eq:Drel})_{2,1}$ of $\mathcal{R}$. Adding $\pm h(\phi)\trH^2$ implies
\begin{align}
	P_3:=&\frac{1}{2}\iQt h(\phi)\trC^2 + h(\psi)\trH^2  - h(\phi)\trC\trH - h(\psi)\trH\trC \dx\ddta\nonumber\\
	=& \frac{1}{2}\iQt h(\phi)\Big[\trC-\trH\Big]^2 + \Big[h(\psi)-h(\phi)\Big]\trH^2 - \Big[h(\psi)-h(\phi)\Big]\trH\trC \dx\ddta\label{eq:P3}\\
	=& \frac{1}{2}\iQt h(\phi)\Big[\trC-\trH\Big]^2 - \Big[h(\psi)-h(\phi)\Big]\trH\Big[\trC - \trH\Big] \dx\ddta
	\leq c\int_0^t  (1+h_2+\norm*{h^\prime}_\infty^2\norm*{\trH}_4^2) \mathcal{E}(\dta) \ddta.\nonumber
\end{align}

Next we treat the terms emitting from diffusion, i.e. $I_{\CC,1}$, together with corresponding dissipation term in \eqref{eq:Drel}, i.e. the sixth and twelfth terms $(\ref{eq:Drel})_{1,6},(\ref{eq:Drel})_{1,12}$
\begin{align}
	-P_4:=\frac{ \varepsilon_2}{2}\iQt \snorm*{\na\CC}^2 + \snorm*{\na\HH}^2 -2\na\CC:\na\HH \dx\ddta  = -\frac{ \varepsilon_2}{2}\iQt \snorm*{\na\CC -\na\HH}^2  \dx\ddta.\label{eq:P4}
\end{align}

Further we consider the relaxation term, i.e. $I_{\CC,2}$, together with the dissipation terms of \eqref{eq:Drel}, i.e the fifth and eleventh terms $(\ref{eq:Drel})_{1,5}, (\ref{eq:Drel})_{1,11}$. Adding first $\pm h(\phi)\trC^2\snorm*{\HH}^2$ and $\pm h(\psi)\trC^2\HH:(\CC-\HH)$ yields
\begin{align}
	P_5:=&-\frac{1}{2}\iQt h(\phi)\trC^2\snorm*{\CC}^2 + h(\psi)\trH^2\snorm*{\HH}^2  - h(\phi)\trC^2\CC:\HH - h(\psi)\trH^2\HH:\CC \dx\ddta\nonumber\\
	=& - \frac{1}{2} \iQt h(\phi)\trC^2(\CC-\HH)^2 + h(\phi)\trC^2(\CC -\HH):\HH + h(\psi)\trH^2(\HH-\CC):\HH \dx\ddta\nonumber\\
	%=& - \frac{1}{2} \iQt h(\phi)\trC^2(\CC-\HH)^2 + \Big[h(\phi)\trC^2 - h(\psi)\trH^2\Big]\HH:\Big[\CC-\HH\Big] \dx\ddta\nonumber\\
	=& - \frac{1}{2} \iQt h(\phi)\trC^2(\CC-\HH)^2 + \Big[h(\phi) - h(\psi)\Big]\trC^2\HH:\Big[\CC-\HH\Big] \dx\ddta\nonumber\\
	&- \frac{1}{2} \iQt h(\psi)\Big[\trC^2 - \trH^2\Big]\HH:\Big[\CC-\HH\Big] \dx\ddta. \label{eq:P5}
\end{align}
To estimate the above terms we rewrite $\trC^2-\trH^2$ as $[\trC-\trH][\trC+\trH]$ and obtain
\begin{align}
	P_5\leq&  -(1-\delta)\frac{1}{2} \iQt h(\phi)\trC^2(\CC-\HH)^2 \dx\ddta + c(h,\delta)\int_0^t \norm*{h^\prime}_\infty^2\norm*{\phi-\psi}_4^2\norm*{\HH}_\infty^2\norm*{\trC}_4^2 \ddta \label{eq:P5est}\\
	&+ c(h,\delta)\int_0^t \norm*{\CC-\HH}_2^2\norm*{\HH}_4^2\norm*{\trC+\trH}_4^2\ddta + \delta\norm*{\na\CC-\na\HH}_{L^2(L^2)}^2\nonumber\\
	\leq&   \iQt -\frac{(1-\delta)}{2}h(\phi)\trC^2(\CC-\HH)^2\dx\ddta + \delta\norm*{\na\CC-\na\HH}_{L^2(L^2)}^2 + c\int_0^t \mathcal{E}(\dta)\(\norm*{\HH}_\infty^2\norm*{\trC}_4^2 + \norm*{\trC+\trH}_4^2 \norm*{\HH}_4^2\) \ddta.  \nonumber
\end{align}

For the convective term of the conformation tensor, i.e. $T_{\CC,4}$, we add $\BB(\u,\CC,\CC)=0$, $\BB(\U,\HH,\HH)=0$ and $\pm\BB(\u,\HH,\CC-\HH)$ to derive
\begin{align}
	P_6:=&\frac{1}{2}\int_0^t\BB(\u,\CC,\HH) + \BB(\U,\HH,\CC) \ddta= -\frac{1}{2}\int_0^t\BB(\u,\CC,\CC-\HH) - \BB(\U,\HH,\CC-\HH) \ddta\label{eq:P6}\\
	=& -\frac{1}{2}\int_0^t\BB(\u,\CC-\HH,\CC-\HH) +  \BB(\u-\U,\HH,\CC-\HH) \ddta= \frac{1}{2}\int_0^t\BB(\u-\U,\CC-\HH,\HH) \ddta \nonumber\\
	\leq& \delta\iQt \snorm*{\na\CC-\na\HH}^2\dx\ddta + \delta\iQt \eta(\phi)\snorm*{\Du-\mathrm{D}\U}^2\dx\ddta + c(\eta)\int_0^t \norm*{\u-\U}_2^2\norm*{\HH}_4^4 \ddta. \nonumber
\end{align}

The last term is $I_{\CC,5}$ which will be treated together with $I_{\u,2}$.
Here we apply Lemma~\ref{lemm:hana} and add the terms $\pm\na\U(\CC-\HH):\HH$,  $\pm(\CC-\HH)(\na\U)^T:\HH$ to obtain
\begin{align}
	P_7:=&-\frac{1}{2}\iQt (\na\u)\CC:\HH + \CC(\na\u)^T:\HH + (\na\U)\HH:\CC + \HH(\na\U)^T:\CC\dx\ddta +\iQt \trC\CC:\na\U + \trH\HH:\na\u  \dx\ddta\nonumber\\
	=& -\frac{1}{2}\iQt (\na\u)\CC:\HH + \CC(\na\u)^T:\HH + (\na\U)\HH:\CC + \HH(\na\U)^T:\CC \dx\ddta\nonumber \\
	&+\frac{1}{2}\iQt   (\na\u)\HH:\HH + \HH(\na\u)^T:\HH + (\na\U)\CC:\CC + \CC(\na\U)^T:\CC \dx\ddta\nonumber\\
	=&-\frac{1}{2}\iQt  \big[(\na\u)(\CC-\HH) \big]:\HH  + \big[(\CC-\HH)(\na\u)^T \big]:\HH \dx\ddta\nonumber \\
	&+\iQt\big[(\na\U)(\CC-\HH) \big]:\CC + \big[(\CC-\HH)(\na\U)^T \big]:\CC \dx\ddta\nonumber \\
	=&-\frac{1}{2}\iQt \big[(\na\u-\na\U)(\CC-\HH) \big]:\HH +  \big[(\CC-\HH)(\na\u-\na\U)^T \big]:\HH \dx\ddta\nonumber\\
	&+\frac{1}{2}\iQt \big[(\na\U)(\CC-\HH) \big]:(\CC-\HH) +  \big[(\CC-\HH)(\na\U)^T \big]:(\CC-\HH) \dx\ddta\label{eq:P7}\\
	\leq& \delta\norm*{\sqrt{\eta(\phi)}(\Du-\mathrm{D}\U)}^2_{L^2(L^2)} + 2\delta\norm*{\na\CC-\na\HH}_{L^2(L^2)}^2 + c(\eta,\delta)\int_0^t  \mathcal{E}(\dta)(\norm*{\na\U}_2^2 + \norm*{\HH}_4^4) \ddta.\nonumber
\end{align}

Summing all estimates for $P_i, i=1,\ldots,7$, i.e. (\ref{eq:P1}), (\ref{eq:P2}), (\ref{eq:P3}), (\ref{eq:P4}), (\ref{eq:P5est}), (\ref{eq:P6}) and (\ref{eq:P7}),  we find
\begin{align}
	\sum_{i=1}^7P_i \leq& (1-4\delta)\iQt \eta(\phi)\snorm*{\Du-\mathrm{D}\U}^2 + \frac{\varepsilon_2}{2}\snorm*{\na\CC-\na\HH}^2 + \frac{1}{2}h(\phi)\trC^2(\CC-\HH)^2 \dx\ddta \label{eq:UPest} \\
	&+ c\int_\Omega \Big(1 + \norm*{\mathrm{D}\U}_3^2 + \norm*{\U}_6^4 + \norm*{\trH}_4^2  + \norm*{\trC+\trH}_4^2 \norm*{\HH}_4^2 + \norm*{\HH}_4^4 + \norm*{\HH}_\infty^2\norm*{\trC}_4^2 \Big) \mathcal{E}(\dta) \ddta. \nonumber
\end{align}
The only remaining terms are the coupling terms $I_{\u,3}$ which we treat later.

%%%%%%%%%%%%%%%%%%%%%%%%%%%%%%%%%%%%%%%%%%%%%%%%%%%%%%%%%%%%%%%%%%%%%%%%%%%%%%%%%%%%%

\subsection{Cahn-Hilliard equation}
In this subsection we consider the integrals arising from (\ref{eq:mixpot}) and find again by a reformulation and applying suitable test functions the following equality
\begin{align}
	&- \iQt F^{\prime\prime}(\psi(\dta))\Big(\phi(\dta)-\psi(\dta)\Big)\frac{\partial \psi(\dta)}{\partial t} + F^\prime(\psi(\dta))\(\frac{\partial \phi(\dta)}{\partial t} + \frac{\partial \psi(\dta)}{\partial t}\) \dx\ddta \equiv \sum_{k=1}^5 I_{F(\phi),k}\nonumber\\
	\equiv& \iQt m(\psi)\na\pi\na\Big(F^\prime(\psi) + F^{\prime\prime}(\psi)(\phi-\psi) \Big)\dx\ddta + \iQt m(\phi)\na\mu\na(F^\prime(\psi)) \dx\ddta\nonumber\\
	&-\iQt n(\psi)\na\Big(A(\psi)Q\Big)\na\Big(F^\prime(\psi) + F^{\prime\prime}(\psi)(\phi-\psi) \Big)\dx\ddta - \iQt n(\phi)\na\Big(A(\phi)q\Big)\na(F^\prime(\psi)) \dx\ddta\nonumber\\
	& +\iQt \U\cdot\na\psi\Big(F^\prime(\psi) + F^{\prime\prime}(\psi)(\phi-\psi)\Big) + \(\u\cdot\na\phi\) F^\prime(\psi) \dx\ddta. \label{eq:potexpand}
\end{align}

Furthermore, we expand (\ref{eq:mixphi}) that yields
\begin{align}
	&c_0\iQt \frac{\partial \phi(\dta)}{\partial t}\Delta \psi(\dta) + \frac{\partial \psi(\dta)}{\partial t}\Delta\phi(\dta) \dx\ddta \equiv\sum_{k=1}^3 I_{\phi,k} \label{eq:phiexpand}\\
	\equiv& -c_0\iQt m(\phi)\na\mu\na\Delta\psi + m(\psi)\na\pi\na\Delta\phi\dx\ddta + c_0\iQt n(\phi)\na\Big(A(\phi)q\Big)\na\Delta\psi  + n(\psi)\na\Big(A(\psi)Q\Big)\na\Delta\phi \dx\ddta\nonumber\\
	& -  c_0\iQt\(\u\cdot\na\phi\)\Delta\psi - \(\U\cdot\na\psi\)\Delta\phi \dx\ddta. \nonumber
\end{align}

Combining (\ref{eq:potexpand}) and (\ref{eq:phiexpand}) and by adding $$\pm m(\psi)\na\pi\na(F^\prime(\phi)), \quad\pm (\U\cdot\na\psi)F^\prime(\phi), \quad \pm n(\psi)\na\Big(A(\psi)Q\Big)\na(F^\prime(\phi))$$ we derive
\begin{align}
 \sum_{k=1}^5 I_{F(\phi),k} + \sum_{k=1}^3 I_{\phi,k}=&\iQt  m(\phi)\na\mu\na\pi + m(\psi)\na\mu\na\pi- m(\psi)\na\pi\na\Big(F^\prime(\phi) - F^\prime(\psi) - F^{\prime\prime}(\psi)(\phi-\psi)\Big) \dx\ddta\nonumber\\
	&-\iQt \U\cdot\na\psi\Big(F^\prime(\phi) - F^\prime(\psi) - F^{\prime\prime}(\psi)(\phi-\psi)\Big)\dx\ddta + \iQt \(\u\cdot\na\phi\)\pi + \(\U\cdot\na\psi\)\mu \dx\ddta\nonumber\\
	&-\iQt n(\phi)\na\Big(A(\phi)q\Big)\na\pi + n(\psi)\na\Big(A(\psi)Q\Big)\na\mu  \dx\ddta\nonumber\\
	&+\iQt n(\psi)\na\Big(A(\psi)Q\Big)\na\Big(F^\prime(\phi) - F^\prime(\psi) - F^{\prime\prime}(\psi)(\phi-\psi) \Big) \dx\ddta. \label{eq:pqmwithouta}
\end{align}

Finally, by expanding the penalty term (\ref{eq:phistab}) we find
\begin{align}
	&2a\iQt  \frac{\partial \phi(\dta)}{\partial t}(\phi(\dta)-\psi(\dta)) + \frac{\partial \psi(\dta)}{\partial t}(\psi(\dta)-\phi(\dta)) \dx\ddta \equiv \sum_{k=1}^3 I_{a,k} \label{eq:penatlyexpand}\\
	\equiv& -2a\iQt m(\phi)\na\mu(\na\phi - \na\psi) + m(\psi)\na\pi(\na\psi - \na\phi) \dx\ddta \nonumber\\
	& +2a\iQt n(\phi)\na\Big(A(\phi)q\Big)(\na\phi - \na\psi) + n(\psi)\na\Big(A(\psi)Q\Big)(\na\psi - \na\phi)  \dx\ddta +2a\iQt \(\u\cdot\na\phi\)\psi + \(\U\cdot\na\psi\)\phi\dx\ddta. \nonumber
\end{align}

\subsection{Equation for the bulk stress}
Due to the cross-diffusive coupling between the Cahn-Hilliard equation $(\ref{eq:weak_reg})_1$ and the bulk stress equation $(\ref{eq:weak_reg})_2$ it is necessary to consider the correction terms arising from $(\ref{eq:Drel})_3$ together.
Therefore we expand (\ref{eq:mixq}) into
\begin{align*}
	&- \iQt \frac{\partial q(\dta)}{\partial t} Q(\dta) + q(\dta)\frac{\partial Q(\dta)}{\partial t} \dx\ddta \\
	=& \iQt \frac{1}{\tau_b(\phi)}qQ + \frac{1}{\tau_b(\psi)}qQ + 2\varepsilon_1\na q\na Q + \na\Big(A(\phi)q\Big)\na\Big(A(\phi)Q\Big) + \na\Big(A(\psi)Q\Big)\na\Big(A(\psi)q\Big) \dx\ddta \\
	&-\iQt n(\phi)\na\mu\na\Big(A(\phi)Q\Big) + n(\psi)\na\pi\na\Big(A(\psi)q\Big) - \(\u\cdot\na q\)Q - \(\U\cdot\na Q\)q \dx\ddta\\
	=& \iQt \frac{2}{\tau_b(\phi)}qQ + \(\frac{1}{\tau_b(\psi)}-\frac{1}{\tau_b(\phi)}\)qQ  + \na\Big(A(\phi)q\Big)\na\Big(A(\phi)Q\Big) + \na\Big(A(\psi)Q\Big)\na\Big(A(\psi)q\Big) \dx\ddta \\
	&-\iQt n(\phi)\na\mu\na\Big(A(\phi)Q\Big) + n(\psi)\na\pi\na\Big(A(\psi)q\Big) - \((\u-\U)\cdot\na q\)Q - 2\varepsilon_1\na q\na Q \dx\ddta \equiv \sum_{k=1}^5 I_{q,k}.
\end{align*}

Considering the relaxation terms $I_{q,1}$ and the corresponding terms in $\tilde{\mathcal{D}}$, i.e. the second and eighths terms $(\ref{eq:Drel})_2, (\ref{eq:Drel})_8$ we find
\begin{align}
	P_8:=&-\iQt \frac{1}{\tau_b(\phi)}q^2 + \frac{1}{\tau_b(\psi)}Q^2 - \frac{2}{\tau_b(\phi)}qQ - \(\frac{1}{\tau_b(\psi)}-\frac{1}{\tau_b(\phi)}\)qQ \dx\ddta \nonumber\\
	=&-\iQt \frac{1}{\tau_b(\phi)}(q-Q)^2\dx\ddta - \iQt \(\frac{1}{\tau_b(\psi)}-\frac{1}{\tau_b(\phi)}\)(Q^2 - qQ) \dx\ddta\nonumber\\
	\leq& -\iQt \frac{1}{\tau_b(\phi)}(q-Q)^2\dx\ddta + \int_0^t \norm*{ (1/\tau_b)^\prime}_\infty\norm*{\phi-\psi}_6\norm*{Q}_3\norm*{q-Q}_2 \ddta\nonumber\\
	\leq& -\iQt \frac{1}{\tau_b(\phi)}(q-Q)^2\dx\ddta + c(\tau_b^\prime)\int_0^t (\norm*{Q}_3^2 + 1)\mathcal{E}(\dta) \ddta. \label{eq:relaxest}
\end{align}

Considering the linear diffusion term $I_{q,5}$ and the related terms in $\tilde{\mathcal{D}}$, i.e. the third and ninth terms $(\ref{eq:Drel})_3, (\ref{eq:Drel})_9$ we obtain
\begin{align}
	P_9:=&-\varepsilon_1\iQt \snorm*{\na q}^2 + \snorm*{\na Q}^2 -2\na q \na Q \dx\ddta= - \varepsilon_1\iQt \snorm*{\na q - \na Q}^2 \dx\ddta. \label{eq:lineardiffusionq}
\end{align}
All the remaining terms will be absorbed either into a nonlinear diffusive remainder $\mathcal{R}_{mix}$ which is associated to the cross-diffusive coupling in (\ref{eq:energyineq}) or to a convective remainder $\mathcal{R}_{conv}$ which collects the remaining convective terms. The  remainders are given by the following equations.

\begin{align}
	\mathcal{R}_{mix} =& -\iQt \Big[n(\phi)\na\mu - \na\big(A(\phi)q\big)\Big]^2 + \Big[n(\psi)\na\pi - \na\big(A(\psi)Q\big)\Big]^2 \dx\ddta +\iQt m(\phi)\na\mu\na\pi + m(\psi)\na\mu\na\pi \dx\ddta\label{eq:rmix1}\\
	&-\iQt n(\phi)\na\mu\na\Big(A(\phi)Q\Big) + n(\psi)\na\pi\na\Big(A(\psi)q\Big) \dx\ddta-\iQt n(\phi)\na\pi\na\Big(A(\phi)q\Big) + n(\psi)\na\mu\na\Big(A(\psi)Q\Big) \dx\ddta\nonumber \\
	&+\iQt \na\Big(A(\phi)q\Big)\na\Big(A(\phi)Q\Big) + \na\Big(A(\psi)Q\Big)\na\Big(A(\psi)q\Big) \dx\ddta-2a\iQt m(\phi)\na\mu(\na\phi - \na\psi) + m(\psi)\na\pi(\na\psi - \na\phi)\dx\ddta\nonumber\\
	&+2a\iQt n(\phi)\na\Big(A(\phi)q\Big)(\na\phi - \na\psi) + n(\psi)\na\Big(A(\psi)Q\Big)(\na\psi - \na\phi)\dx\ddta, \nonumber\\
	\mathcal{R}_{conv} =&  \iQt \(\u\cdot\na\phi\)\pi + \(\U\cdot\na\psi\)\mu \dx\ddta- \iQt\mu\na\phi\cdot\U + \pi\na\psi\cdot\u\dx\ddta + \iQt \((\u-\U)\cdot\na q\)Q\dx\ddta   \nonumber\\
	&+ 2a\iQt \((\u-\U)\cdot\na\phi\)\psi\dx\ddta-\iQt \(\U\cdot\na\psi\)\Big(F^\prime(\phi) - F^\prime(\psi) - F^{\prime\prime}(\psi)(\phi-\psi)\Big)\dx\ddta \nonumber\\
	&- \iQt \Big[m(\psi)\na\pi -  n(\psi)\na\Big(A(\psi)Q\Big)\Big]\na\Big(F^\prime(\phi) - F^\prime(\psi) - F^{\prime\prime}(\psi)(\phi-\psi)\Big) \dx\ddta. \label{eq:rconv1}
\end{align}

\subsection{Convective remainder term}%$\mathcal{R}_{conv}$
In this subsection our aim is to  estimate the term $\mathcal{R}_{conv}$, cf. (\ref{eq:rconv1}).
We apply integration by parts, use (\ref{eq:bc}), and add $\pm\pi\U(\na\phi-\na\psi)$, $+(\u-\U)\na\psi\psi$, $-(\u-\U)\na QQ$.
Further,  we add the terms $\pm\U(\phi-\psi)\na\Big(A(\phi)(q-Q)\Big)/n(\phi)$  to obtain the following representation of $\mathcal{R}_{conv}$
\begin{align}
	\mathcal{R}_{conv} =&  \iQt \pi\u\cdot(\na\phi-\na\psi) + \mu\U\cdot(\na\psi-\na\phi)\dx\ddta  -\iQt \((\u-\U)\cdot\na Q\)q\dx\ddta + 2a\iQt \((\u-\U)\cdot\na\phi\)\psi\dx\ddta\nonumber\\
	&- \frac{1}{2}\iQt \mathrm{div}\Big[\U\psi - m(\psi)\na\pi +  n(\psi)\na\Big(A(\psi)Q\Big)\Big]\Big( F^{\prime\prime\prime}(z)(\phi-\psi)^2\Big) \dx\ddta\nonumber\\
	=&  \iQt \pi(\u-\U)\cdot(\na\phi-\na\psi) + (\na\mu-\na\pi)\U(\phi-\psi) \dx\ddta - 2a\iQt (\u-\U)\cdot\na\psi(\phi-\psi)\dx\ddta \nonumber\\
	& - \iQt (\u-\U)\cdot\na Q(q-Q) \dx\ddta- \frac{1}{2}\iQt \mathrm{div}\Big[\U\psi - m(\psi)\na\pi +  n(\psi)\na\Big(A(\psi)Q\Big)\Big]\Big( F^{\prime\prime\prime}(z)(\phi-\psi)^2\Big) \dx\ddta\nonumber\\
	=&  \iQt -\na\pi\cdot(\u-\U)(\phi-\psi) + \U\frac{\phi-\psi}{n(\phi)}\Big[n(\phi)(\na\mu-\na\pi) - \na\Big(A(\phi)(q-Q) \Big) \Big] \dx\ddta  \nonumber\\
	&- 2a\iQt (\u-\U)\cdot\na\psi(\phi-\psi)\dx\ddta - \iQt (\u-\U)\cdot\na Q(q-Q) \dx\ddta  +\iQt \U\frac{\phi-\psi}{ n(\phi)}\na\Big(A(\phi)(q-Q)\Big) \dx\ddta \label{eq:rconvfin}\\
	&- \frac{1}{2}\iQt \mathrm{div}\Big[\U\psi - m(\psi)\na\pi +  n(\psi)\na\Big(A(\psi)Q\Big)\Big]\Big( F^{\prime\prime\prime}(z)(\phi-\psi)^2\Big) \dx\ddta .\nonumber
\end{align}
 Here $z$ denotes a suitable convex combination of $\phi$ and $\psi$.
We will now estimate each term of (\ref{eq:rconvfin}) separately and obtain
\begin{align*}
	I_1 &\leq \int_0^t \norm*{\na\pi}_3\norm*{\u-\U}_2\norm*{\phi-\psi}_6 \ddta \leq c\int_0^t \mathcal{E}(\dta)(1+\norm*{\na\pi}_3^2) \ddta, \\
	I_2 &\leq \delta\norm*{n(\phi)(\na\mu-\na\pi) - \na\Big(A(\phi)(q-Q) \Big)}^2_{L^2(L^2)} + c\int_0^t \norm*{\U}_3^2\norm*{\phi-\psi}_6^2 \ddta,\\
	I_3 &\leq \int_0^t \norm*{\na\psi}_3\norm*{\u-\U}_2\norm*{\phi-\psi}_6 \ddta \leq c\int_0^t \mathcal{E}(\dta)(1+\norm*{\na\psi}_3^2) \ddta,\\
	I_4 &\leq \int_0^t \norm*{\na Q}_6\norm*{\u-\U}_2\norm*{q-Q}_3 \ddta \leq \delta\norm*{\na q-\na Q}_{L^2(L^2)}^2 + c\int_0^t \mathcal{E}(\dta)(1+\norm*{\na Q}_6^4) \ddta,\\
	I_5 &\leq c\int_0^t \norm*{\U}_\infty\norm*{\phi-\psi}_4\norm*{q-Q}_4\norm*{\na\phi}_2 + \norm*{\U}_\infty\norm*{\na q-\na Q}_2\norm*{\phi-\psi}_2 \ddta \\
	&\leq c\int_0^t \norm*{\U}_\infty\norm*{\phi-\psi}_6\norm*{q-Q}_{H^1}\norm*{\na\phi}_3 + \norm*{\U}_\infty\norm*{\na q-\na Q}_2\norm*{\phi-\psi}_2 \ddta \\
	&\leq 2\delta\norm*{\na q-\na Q}_{L^2(L^2)}^2 + c\int_0^t \mathcal{E}(\dta)(1+\norm*{\na\phi}_3^2\norm*{\U}_\infty^2 + \norm*{\U}_\infty^2) \ddta,\\
	I_6 &\leq \int_0^t \norm*{\div{\U\psi-m(\psi)\na\pi+n(\psi)\na\Big(A(\psi)Q\Big)}}_2\norm*{F^{\prime\prime\prime}(z)}_3\norm*{\phi-\psi}_6^2 \ddta \\
	&\leq c\int_0^t  \mathcal{E}(\dta)\norm*{\div{\U\psi-m(\psi)\na\pi+n(\psi)\na\Big(A(\psi)Q\Big)}}_2(1 + \norm*{\phi}_3 + \norm*{\psi}_3) \ddta,
\end{align*}
where $z$ is a suitable convex combination of $\phi$ and $\psi.$
Here we have used Assumptions \ref{ass:regular} to control $\norm*{F^{\prime\prime\prime}(z)}_3$ by $c(\norm{\phi}_3 + \norm{\psi}_3).$
Summing all the estimates we find
\begin{align}
	\sum_{i=1}^6 I_i \leq& \delta\norm*{n(\phi)(\na\mu-\na\pi) - \na\Big(A(\phi)(q-Q) \Big)}^2_{L^2(L^2)} + 3\delta\norm*{\na q-\na Q}_{L^2(L^2)}^2 + c\int_0^t \Big( 1+\norm*{\na\pi}_3^2 + \norm*{\U}_3^2 + \norm*{\na\psi}_3^2  \label{eq:conv}\\
	& + \norm*{\na Q}_6^4  + \norm*{\na\phi}_3^2\norm*{\U}_\infty^2 + \norm*{\U}_\infty^2+ \norm*{\div{\U\psi-m(\psi)\na\pi+n(\psi)\na\Big(A(\psi)Q\Big)}}_2(1 + \norm*{\phi}_3 + \norm*{\psi}_3)\Big)  \mathcal{E}(\dta)\ddta. \nonumber
\end{align}

\subsection{Cross-diffusion remainder term}%$\mathcal{R}_{mix}
This section is devoted to the treatment of (\ref{eq:rmix1}). For convenience we split the integral further into $\mathcal{\tilde{R}}_{mix}$ which contains all diffusive terms without the penalty term and a penalty remainder $\mathcal{R}_a$
\begin{equation}
	\mathcal{R}_{mix} = \mathcal{\tilde{R}}_{mix} + \mathcal{R}_a \label{eq:splitar}.
\end{equation}
We will first consider $\mathcal{R}_{mix}$ given by
\begin{align*}
	\mathcal{\tilde{R}}_{mix} =& -\iQt \Big|n(\phi)(\na\mu-\na\pi) - \na\Big(A(\phi)(q-Q)\Big) \Big|^2 \dx\ddta  +\iQt \Big(m(\psi) - m(\phi)\Big)(\na\mu-\na\pi)\na\pi\dx\ddta \\
	& +\iQt n(\phi)\na\mu\na\Big(A(\phi)Q\Big)\dx\ddta - \iQt n(\psi)\na\mu\na\Big(A(\psi)Q\Big)\dx\ddta  +\iQt n(\phi)\na\pi\na\Big(A(\phi)q\Big)\dx\ddta  \\
	& - \iQt n(\psi)\na\pi\na\Big(A(\psi)q\Big)\dx\ddta-2\iQt n(\phi)\na\pi\na\Big(A(\phi)Q\Big)\dx\ddta  + 2\iQt n(\psi)\na\pi\na\Big(A(\psi)Q\Big)\dx\ddta \\
	& -\iQt \snorm*{\na\Big(A(\psi)Q\Big)}^2 + \iQt \snorm*{\na\Big(A(\phi)Q\Big)}^2\dx\ddta  -\iQt \na\Big(A(\phi)q\Big)\na\Big(A(\phi)Q \Big)\dx\ddta + \iQt \na\Big(A(\psi)q\Big)\na\Big(A(\psi)Q\Big) \dx\ddta.
\end{align*}
After simple but tedious calculations we obtain the following result.
\begin{Lemma}
	$\mathcal{\tilde{R}}_{mix}$ can be rewritten as
	\begin{align*}
		\mathcal{\tilde{R}}_{mix} =& -\iQt \Big|n(\phi)(\na\mu-\na\pi) - \na\Big(A(\phi)(q-Q)\Big) \Big|^2 \dx\ddta  + \iQt \Big(n(\psi) - n(\phi)\Big)\na\pi\Big[ n(\phi)(\na\mu-\na\pi) - \na\Big(A(\phi)(q-Q)\Big)\Big] \dx\ddta\\
		& + \iQt \na\Big((A(\phi)-A(\psi))Q\Big)\Big[n(\phi)(\na\mu-\na\pi) - \na\Big(A(\phi)(q-Q)\Big)\Big]\dx\ddta \\
		& + \iQt \Big(\frac{n(\psi) - n(\phi)}{n(\phi)}\Big)\Big[n(\phi)(\na\mu-\na\pi) - \na\Big(A(\phi)(q-Q)\Big)\Big]\Big[ n(\psi)\na\pi - \na\Big(A(\psi)Q\Big) \Big]\dx\ddta \\
		& + \iQt \Big(n(\psi) - n(\phi)\Big)\Big[n(\psi)\na\pi - \na\Big(A(\psi)Q \Big) \Big]\na\Big(A(\phi)(q-Q)\Big)/n(\phi) \dx\ddta\\
		& - \iQt \Big((A(\phi)-A(\psi))(q-Q)\Big)\mathrm{div}\Big[n(\psi)\na\pi - \na\Big(A(\psi)Q\Big) \Big]\dx\ddta.
	\end{align*}
	
\end{Lemma}

The proof can be found in the Appendix. We will now estimate all the integral separately which yields
\begin{align*}
	I_0 &= -\norm*{n(\phi)(\na\mu-\na\pi) - \na\Big(A(\phi)(q-Q)\Big)}_{L^2(L^2)}^2, \\
	I_1 &\leq \delta\norm*{n(\phi)(\na\mu-\na\pi) - \na\Big(A(\phi)(q-Q)\Big)}_{L^2(L^2)}^2 + c\int_0^t \norm*{n^\prime}_\infty^2\norm*{\na\pi}_3^2\norm*{\phi-\psi}_6^2\ddta, \\
	I_2 &\leq \delta\norm*{n(\phi)(\na\mu-\na\pi) - \na\Big(A(\phi)(q-Q)\Big)}_{L^2(L^2)}^2 + c\int_0^t \norm*{\na\Big((A(\phi)-A(\psi))Q\Big)}_2^2 \ddta, \\
	I_3 &\leq \delta\norm*{n(\phi)(\na\mu-\na\pi) - \na\Big(A(\phi)(q-Q)\Big)}_{L^2(L^2)}^2 + c\int_0^t \norm*{n^\prime}_\infty^2\norm*{n(\psi)\na\pi-\na\Big(A(\psi) Q\Big)}_3^2\norm*{\phi-\psi}_6^2\ddta, \\
	I_4 &\leq c\int_0^t \norm*{[n(\psi)\na\pi - \na\Big(A(\psi)Q \Big) }_4\(\norm*{\phi-\psi}_6\norm*{\na q - \na Q}_2 + \norm*{\phi-\psi}_6 \norm*{\na\phi}_4\norm*{q-Q}_3 \)\\
	&\leq 2\delta\norm*{\na q-\na Q}^2_{L^2(L^2)} + c\int_0^t \mathcal{E}(\dta)\norm*{[n(\psi)\na\pi - \na\Big(A(\psi)Q \Big) }_4^2\(1 +  \norm*{\na\phi}_4^2  \) \ddta, \\
	I_5 &\leq c\int_0^t \norm*{\mathrm{div}\Big[n(\psi)\na\pi - \na\Big(A(\psi)Q\Big)\Big]}_2\norm*{\phi-\psi}_6\norm*{q-Q}_3 \ddta \\
	%&\leq c\int_0^t \norm*{\mathrm{div}\Big[n(\psi)\na\pi - \na\Big(A(\psi)Q\Big)\Big]}_2\norm*{\phi-\psi}_6\norm*{q-Q}_{H^1} \ddta    \\
	&\leq \delta\norm*{\na q-\na Q}^2_{L^2(L^2)}+ c\int_0^t  \(1+\norm*{\mathrm{div}\Big[n(\psi)\na\pi - \na\Big(A(\psi)Q\Big)\Big]}_2^2\)\mathcal{E}(\dta) \ddta.
\end{align*}

Summing up all above estimates for $\mathcal{\tilde{R}}_{mix}$ and applying Lemma \ref{lem:easy_control} we conclude
\begin{align}
	\sum_{i=0}^5 I_i \leq& -(1-3\delta)\norm*{n(\phi)(\na\mu-\na\pi) - \na\Big(A(\phi)(q-Q)\Big)}_{L^2(L^2)}^2 + 3\delta\norm*{\na q - \na Q}_{L^2(L^2)}^2\nonumber\\
	& +c\int_0^t \Big(\norm*{\na\pi}_3^2 + \norm*{\na Q}_3^2 + \norm*{Q}_\infty^2 + \norm*{\na\phi}_3^2\norm*{Q}_\infty^2 + \norm*{\na\psi}_3^2\norm*{Q}_\infty^2 \Big. \label{eq:Rmixx}\\
	&+ \left.\norm*{n(\psi)\na\pi - \na\Big(A(\psi)Q \Big) }_4^2\(1 + \norm*{\na\phi}_4^2  \) +\norm*{\mathrm{div}\Big[n(\psi)\na\pi - \na\Big(A(\psi)Q\Big)\Big]}_2^2 \)\mathcal{E}(\dta) \ddta. \nonumber
\end{align}

Finally, we consider the remainder $\mathcal{R}_a$, see (\ref{eq:splitar}), by adding the terms $\pm2a m(\phi)\na\pi\na(\phi-\psi)$, $\pm2a n(\phi)\na\Big(A(\phi)Q\Big)\na(\phi-\psi)$, $\pm2a n(\psi)\na\Big(A(\phi)Q\Big)\na(\phi-\psi)$  leads to
\begin{align*}
	\mathcal{R}_a=&-2a\iQt m(\phi)\na\mu(\na\phi - \na\psi) + m(\psi)\na\pi(\na\psi - \na\phi)\dx\ddta\\
	&+2a\iQt n(\phi)\na\Big(A(\phi)q\Big)(\na\phi - \na\psi) + n(\psi)\na\Big(A(\psi)Q\Big)(\na\psi - \na\phi)\dx\ddta \\
	=& -2a\iQt n(\phi)\Big[n(\phi)(\na\mu-\na\pi) - \na\Big(A(\phi)(q-Q)\Big) \Big]\na(\phi-\psi) \dx\ddta\\
	&- 2a\iQt \Big(m(\phi) - m(\psi)\Big)\na\pi\na(\phi-\psi) \dx\ddta+ 2a\iQt \Big(n(\phi) - n(\psi)\Big)\na\Big(A(\phi)Q\Big)\na(\phi-\psi)\dx\ddta \\
	&+ 2a\iQt n(\psi)\na\Big((A(\phi)-A(\psi))Q\Big)\na(\phi-\psi)\dx\ddta.
\end{align*}
Similarly as in (\ref{eq:Rmixx}) we estimate the integrals of $\mathcal{R}_a$
\begin{align*}
	\mathcal{R}_a&\leq \delta\norm*{[n(\phi)(\na\mu-\na\pi) - \na\Big(A(\phi)(q-Q)\Big)}_{L^2(L^2)}^2  +c\int_0^t \norm*{\na\phi-\na\psi}_2^2 + \norm*{\phi-\psi}_6^2\norm*{\na\pi}_3^2 \ddta \\
	& +c\int_0^t \norm*{\na\phi-\na\psi}_2^2 + \norm*{\psi-\psi}_6^2\norm*{\na\phi}_3^2\norm*{Q}_\infty^2 + \norm*{\psi-\psi}_6^2\norm*{\na Q}_3^2 \ddta  +c\int_0^t \norm*{\na\Big((A(\phi)-A(\psi))Q\Big)}_2\norm*{\na\phi-\na\psi}_2 \ddta .
\end{align*}
Application of Lemma \ref{lem:easy_control} implies
\begin{align}
	\mathcal{R}_a\leq& \;\delta\norm*{[n(\phi)(\na\mu-\na\pi) - \na\Big(A(\phi)(q-Q)\Big)}_{L^2(L^2)}^2 \nonumber\\
	&+ c\int_0^t \mathcal{E}(\dta)\left[1 + \norm*{\na\pi}_3^2 + \norm*{\na\phi}_3^2\norm*{Q}_\infty^2 + \norm*{\na Q}_3^2 +  \norm*{Q}_\infty^2 + \norm*{Q}_\infty^2 \norm*{\na\psi}_3^2 \right] \ddta. \label{eq:Rmixa}
\end{align}

\subsection{Gronwall-type estimates}
In this section we will collect all estimates derived in the previous sections to deduce inequality (\ref{eq:relgron}).
Summing the estimates (\ref{eq:UPest}), (\ref{eq:relaxest}), (\ref{eq:lineardiffusionq}), (\ref{eq:conv}), (\ref{eq:Rmixx}), (\ref{eq:Rmixa}) yields
\begin{align}
	\mathcal{E}(t) + b\mathcal{D} &\leq \mathcal{E}(0)  + \int_0^t g(\dta)\mathcal{E}(\dta) \ddta, \label{eq:rel_gron_a}\\
	\int_0^t \mathcal{E}(\dta)g(\dta)\ddta = &c\int_0^t \mathcal{E}(\dta)\Big[ 1 +
	\norm*{\mathrm{D}\U}_3^2  + \norm*{\U}_6^4 +\norm*{\HH}_\infty^2\norm*{\trC}_4^2 + \norm*{\HH}_4^4 +\norm*{\trH}_4^2 + \norm*{\trC+\trH}_4^2 \norm*{\HH}_4^2\nonumber \\
	&   + \norm*{\na\psi}_3^2 + \norm*{\U}_\infty^2  +\norm*{\na\pi}_3^2 + \norm*{\na\phi}_3^2\norm*{Q}_\infty^2 + \norm*{\na Q}_6^4 +  \norm*{Q}_\infty^2 + \norm*{Q}_\infty^2 \norm*{\na\psi}_3^2 + \norm*{\na\phi}_3^2\norm*{\U}_\infty^2 \label{eq:rel_gron}\\
	& +\norm*{n(\psi)\na\pi - \na\big(A(\psi)Q \big) }_4^2\(1 + \norm*{\na\phi}_4^2 \) +\norm*{\mathrm{div}\big[n(\psi)\na\pi - \na\big(A(\psi)Q\big)\big]}_2^2\nonumber \\
	& + \norm*{\div{\U\psi-m(\psi)\na\pi+n(\psi)\na\big(A(\psi)Q\big)}}_2(\norm*{\phi}_3 + \norm*{\psi}_3)\Big] \ddta. \nonumber
\end{align}
Here $\mathcal{D}$ is given by (\ref{eq:Dfull}) and $b\leq 1-6\delta$. Thus, by choosing $\delta$ small enough we obtain $b < 1$.

Furthermore, the constant $c$ depends on the upper and lower bounds for the parametric functions, the $L^\infty$ norms of their derivatives, on the penalty coefficient $a$ and on the size of the domain $\Omega$.
We can see that for a more regular solution $(\psi,Q,\pi,\U,\HH)$, cf. Definition \ref{defn:moreweak}, $g\in L^1(0,T^\dagger)$ and therefore we can apply the Gronwall lemma which yields Theorem \ref{theo:wsu}.

\section{Conclusion}
We have developed a concept of the relative energy for the viscoelastic phase separation model (\ref{eq:full_model}). Due to the non-convexity of the potential $F$ it was necessary to modify the standard approach by introducing a suitable penalty term.  We showed that the relative energy leads to the weak-strong uniqueness principle for the weak solution \eqref{reg1}, (\ref{eq:weak_reg}).
In other words, if the strong solution satisfying \eqref{eq:wsutimereg} exists then all weak solutions coincide with the corresponding strong solution.

In future work we would like to apply the relative energy inequality in order to analyse convergence of numerical methods applied to viscoelastic phase separation model \eqref{eq:full_model}.

\section*{Acknowledgment}
This research was supported by the German Science Foundation (DFG) under the Collaborative Research Center TRR~146 Multiscale Simulation Methods for Soft Matters (Project~C3). M.L. is grateful to the Gutenberg Research College for supporting her research. We would like to thank A. Schömer for careful proof reading of the manuscript.

\section{Appendix}
\begin{align*}
	\mathcal{\tilde{R}}_{mix} =& -\iQt \Big[n(\phi)(\na\mu-\na\pi) - \na\Big(A(\phi)(q-Q)\Big) \Big]^2 \dx\ddta  +\iQt \Big(m(\psi) - m(\phi)\Big)(\na\mu-\na\pi)\na\pi\dx\ddta \\
	& +\iQt n(\phi)\na\mu\na\Big(A(\phi)Q\Big)\dx\ddta - \iQt n(\psi)\na\mu\na\Big(A(\psi)Q\Big)\dx\ddta +\iQt n(\phi)\na\pi\na\Big(A(\phi)q\Big)\dx\ddta\\
	&  - \iQt n(\psi)\na\pi\na\Big(A(\psi)q\Big)\dx\ddta  -2\iQt n(\phi)\na\pi\na\Big(A(\phi)Q\Big)\dx\ddta  + 2\iQt n(\psi)\na\pi\na\Big(A(\psi)Q\Big)\dx\ddta \\
	& -\iQt \snorm*{\na\Big(A(\psi)Q\Big)}^2 + \iQt \snorm*{\na\Big(A(\phi)Q\Big)}^2\dx\ddta  -\iQt \na\Big(A(\phi)q\Big)\na\Big(A(\phi)Q \Big)\dx\ddta + \iQt \na\Big(A(\psi)q\Big)\na\Big(A(\psi)Q\Big) \dx\ddta.
\end{align*}

Suitable reformulation yields
\begin{align*}
	\mathcal{\tilde{R}}_{mix} =& -\iQt \Big[n(\phi)(\na\mu-\na\pi) - \na\Big(A(\phi)(q-Q)\Big) \Big]^2\dx\ddta  +\iQt \Big(n(\psi) - n(\phi)\Big)\Big(n(\psi) + n(\phi)\Big)(\na\mu-\na\pi)\na\pi\dx\ddta \\
	& +\iQt n(\phi)(\na\mu-\na\pi)\na\Big(A(\phi)Q\Big)\dx\ddta - \iQt n(\psi)(\na\mu-\na\pi)\na\Big(A(\psi)Q\Big)\dx\ddta \\
	& +\iQt n(\phi)\na\pi\na\Big(A(\phi)(q-Q) \Big)\dx\ddta - \iQt n(\psi)\na\pi\na\Big(A(\psi)(q-Q) \Big)\dx\ddta \\
	& -\iQt \na\Big(A(\phi)Q\Big)\na\Big(A(\phi)(q-Q) \Big)\dx\ddta + \iQt \na\Big(A(\psi)Q\Big)\na\Big(A(\psi)(q-Q)\Big) \dx\ddta.
\end{align*}
We add $\pm n(\psi)(\na\mu-\na\pi)\na\Big(A(\phi)Q\Big)$ and $\pm n(\psi)\na\pi\na\Big(A(\phi)(q-Q)\Big)$ to find
\begin{align*}
	\mathcal{\tilde{R}}_{mix} =& -\iQt \Big[n(\phi)(\na\mu-\na\pi) - \na\Big(A(\phi)(q-Q)\Big) \Big]^2 \dx\ddta +\iQt \Big(n(\psi) - n(\phi)\Big)\Big(n(\psi) + n(\phi)\Big)(\na\mu-\na\pi)\na\pi\dx\ddta \\
	& +\iQt \Big(n(\phi) - n(\psi)\Big)(\na\mu-\na\pi)\na\Big(A(\phi)Q\Big) \dx\ddta +\iQt n(\psi)(\na\mu-\na\pi)\na\Big((A(\phi) - A(\psi))Q\Big)\dx\ddta \\
	& +\iQt \Big(n(\phi) - n(\psi)\Big)\na\pi\na\Big(A(\phi)(q-Q) \Big)\dx\ddta +\iQt n(\psi)\na\pi\na\Big((A(\phi) -A(\psi))(q-Q) \Big) \dx\ddta\\
	& -\iQt \na\Big(A(\phi)Q\Big)\na\Big(A(\phi)(q-Q) \Big)\dx\ddta + \iQt \na\Big(A(\psi)Q\Big)\na\Big(A(\psi)(q-Q)\Big) \dx\ddta.
\end{align*}
First we add $\pm\na\Big(A(\psi)Q \Big)\na\Big(A(\phi)(q-Q)\Big)$  and by rearranging the terms we deduce
\begin{align*}
	\mathcal{\tilde{R}}_{mix} =& -\iQt \Big[n(\phi)(\na\mu-\na\pi) - \na\Big(A(\phi)(q-Q)\Big) \Big]^2 \dx\ddta + \iQt \Big(n(\psi) - n(\phi)\Big)\na\pi\Big[ n(\phi)(\na\mu-\na\pi) - \na\Big(A(\phi)(q-Q)\Big)\Big] \dx\ddta\\
	& + \iQt \na\Big((A(\phi)-A(\psi))Q\Big)\Big[n(\psi)(\na\mu-\na\pi) - \na\Big(A(\phi)(q-Q)\Big)\Big] \dx\ddta\\
	& + \iQt \Big(n(\psi) - n(\phi)\Big)(\na\mu-\na\pi)\Big[ n(\psi)\na\pi - \na\Big(A(\phi)Q\Big) \Big]\dx\ddta \\
	& + \iQt \na\Big((A(\phi)-A(\psi))(q-Q)\Big)\Big[n(\psi)\na\pi - \na\Big(A(\psi)Q\Big) \Big]\dx\ddta.
\end{align*}
Now we add $\pm \na\Big((A(\phi)-A(\psi))Q\Big)n(\phi)(\na\mu-\na\pi)$ to derive
\begin{align*}
	\mathcal{\tilde{R}}_{mix} =& -\iQt \Big[n(\phi)(\na\mu-\na\pi) - \na\Big(A(\phi)(q-Q)\Big) \Big]^2 \dx\ddta + \iQt \Big(n(\psi) - n(\phi)\Big)\na\pi\Big[ n(\phi)(\na\mu-\na\pi) - \na\Big(A(\phi)(q-Q)\Big)\Big] \dx\ddta\\
	& + \iQt \na\Big((A(\phi)-A(\psi))Q\Big)\Big[n(\phi)(\na\mu-\na\pi) - \na\Big(A(\phi)(q-Q)\Big)\Big] \dx\ddta\\
	& + \iQt (n(\psi)-n(\phi))(\na\mu-\na\pi)\na\Big((A(\phi)-A(\psi))Q \Big) \dx\ddta\\
	& + \iQt \Big(n(\psi) - n(\phi)\Big)(\na\mu-\na\pi)\Big[ n(\psi)\na\pi - \na\Big(A(\phi)Q\Big) \Big]\dx\ddta \\
	& + \iQt \na\Big((A(\phi)-A(\psi))(q-Q)\Big)\Big[n(\psi)\na\pi - \na\Big(A(\psi)Q\Big) \Big]\dx\ddta.
\end{align*}
Finally, adding $$\pm (n(\psi) - n(\phi))\na\Big( A(\phi)(q-Q) \Big)\Big[n(\psi)\na\pi - \na\Big(A(\psi)Q\Big)\Big]/n(\phi)$$
implies the desired result
\begin{align*}
	\mathcal{\tilde{R}}_{mix} =& -\iQt \Big[n(\phi)(\na\mu-\na\pi) - \na\Big(A(\phi)(q-Q)\Big) \Big]^2 \dx\ddta  + \iQt \Big(n(\psi) - n(\phi)\Big)\na\pi\Big[ n(\phi)(\na\mu-\na\pi) - \na\Big(A(\phi)(q-Q)\Big)\Big] \dx\ddta\\
	& + \iQt \na\Big((A(\phi)-A(\psi))Q\Big)\Big[n(\phi)(\na\mu-\na\pi) - \na\Big(A(\phi)(q-Q)\Big)\Big]\dx\ddta \\
	& + \iQt \frac{n(\psi) - n(\phi)}{n(\phi)}\Big[n(\phi)(\na\mu-\na\pi) - \na\Big(A(\phi)(q-Q)\Big)\Big]\Big[ n(\psi)\na\pi - \na\Big(A(\psi)Q\Big) \Big]\dx\ddta \\
	& + \iQt \frac{n(\psi) - n(\phi)}{n(\phi)}\Big[n(\psi)\na\pi - \na\Big(A(\psi)Q \Big) \Big]\na\Big(A(\phi)(q-Q)\Big) \dx\ddta\\
	& - \iQt \Big((A(\phi)-A(\psi))(q-Q)\Big)\mathrm{div}\Big[n(\psi)\na\pi - \na\Big(A(\psi)Q\Big) \Big]\dx\ddta.
\end{align*}

\bibliography{paper_bib.bib}
\bibliographystyle{abbrv}

\end{document}